\documentclass[final,hidelinks,onefignum,onetabnum]{siamart250106}

\usepackage{amsfonts}
\usepackage{graphicx}
\usepackage{epstopdf}
\usepackage{algorithmic}

\usepackage{amsopn}

\newcommand{\SCAL}[3][]
{%
  \ifthenelse{\equal{#1}{}}
  {\left\langle#2,#3\right\rangle}
  {\left\langle#2,#3\right\rangle_{#1}}
}
\newcommand{\NORM}[2][]
{%
  \ifthenelse{\equal{#1}{}}
  {\left\|#2\right\|}
  {\left\|#2\right\|_{#1}}
}

\newcommand{\R}{\mathbb{R}}

\newcommand{\Gc}{\mathcal{G}}

\newcommand{\length}{\mathrm{length}}

\newcommand{\Hc}{\mathcal{H}}

\newcommand{\rayl}{\mathcal{R}}

\newcommand{\eigenfun}{f}

\newcommand{\Span}{\operatorname{span}}

\newcommand{\grad}[1][]{
  \ifthenelse{\equal{#1}{}}
  {\nabla_{\edgelength}}
  {\nabla_{#1}}}

\newcommand{\Op}[2][]
{%
  \ifthenelse{\equal{#1}{}}
  {\left[#2\right]}
  {\left[#2\right]_{#1}}
}

\newcommand{\Opc}[2][]
{%
  \ifthenelse{\equal{#1}{}}
  {\left\{#2\right\}}
  {\left\{#2\right\}_{#1}}
}

\newcommand{\locgrad}[2][]
{%
  \ifthenelse{\equal{#1}{}}
  {\Opc{\grad #2}}
  {\Opc[#1]{\grad #2}}
}

\newcommand{\divg}{\operatorname{div}}
\newcommand{\Div}[2][]
{%
  \ifthenelse{\equal{#1}{}}
  {{\divg #2}}
  {{\divg #2}_{#1}}
}

\newcommand{\plap}{\Delta_p}

\newcommand{\inflap}{\Delta_{\infty}}

\newcommand{\lap}[1][]{
  \ifthenelse{\equal{#1}{}}
  {\Delta}
  {\Delta_{#1}}}

\newcommand{\lyap}[1][]{
  \ifthenelse{\equal{#1}{}}
  {\mathcal{L}}
  {\mathcal{L}_{#1}}}

\newcommand{\edgeset}{E}
\newcommand{\nodeset}{V}
\newcommand{\boundary}{B}

\newcommand{\internalnodes}{\nodeset_{\text{I}}}

\newcommand{\edgemaxset}{E_{\text{max}}}

\newcommand{\nodemaxset}{V_{\text{max}}}

\newcommand{\edgelength}{\omega}

\newcommand{\edgeweight}[1][]{
  \ifthenelse{\equal{#1}{}}
  {\mu}
  {\mu_{#1}}}
\newcommand{\nodeweight}[1][]{
  \ifthenelse{\equal{#1}{}}
  {\nu}
  {\nu_{#1}}}

\DeclareMathOperator*{\argmax}{arg\:max}
\DeclareMathOperator*{\argmin}{arg\:min}

\newcommand{\morse}[1][]{
  \ifthenelse{\equal{#1}{}}
  {\mathcal{MI}}
  {\mathcal{MI}_{#1}}}

\newcommand{\new}[1]{\textcolor{black}{#1}}

\newenvironment{NEW}
{\color{black}
}
{ 
  \color{black}
}

\newcommand{\To}{\rightarrow}
\usepackage[english]{babel}

\usepackage{color}
\usepackage{amsmath, amssymb, mathdots, accents}
\usepackage{ifthen}

\usepackage{float}
\usepackage{mathtools}

\usepackage{soul}

\usepackage{tikz, tikz-3dplot}
\usetikzlibrary{backgrounds}

\ifpdf
\DeclareGraphicsExtensions{.eps,.pdf,.png,.jpg}
\else
\DeclareGraphicsExtensions{.eps}
\fi

\newsiamremark{remark}{Remark}
\newsiamremark{example}{Example}

\headers{The graph $\infty$-Laplacian eigenvalue problem}{P. Deidda, M. Burger, M. Putti, and F. Tudisco}

\title{The graph $\infty$-Laplacian eigenvalue problem\thanks{\funding{Piero Deidda and Mario Putti are members of the Gruppo Nazionale Calcolo Scientifico-Istituto Nazionale di Alta Matematica (GNCS-INdAM).
Mario Putti was partially supported by "iNEST: Interconnected Nord-Est Innovation Ecosystem” funded under the National Recovery and Resilience Plan (NRRP), Mission 4 Component 2 Investment 1.5 - Call for tender No. 3277 of 30 december 2021 of Italian Ministry of University and Research funded by the European Union – NextGenerationEU, Project code: ECS00000043, Concession Decree No. 1058 of June 23, 2022, CUP C43C22000340006.
Piero Deidda was supported by a grant from Fondazione Ing. Aldo Gini, Padova. Piero Deidda and Francesco Tudisco were supported by the MUR-PRO3 grant STANDS. Martin Burger thanks the DFG for support via the SFB TRR 154, Subproject C06. Martin Burger acknowledges
support from DESY (Hamburg, Germany), a member of the Helmholtz Association HGF.}}}

\author{Piero Deidda\thanks{Scuola Normale Superiore, Piazza dei Cavalieri, 7, 56126 Pisa, Italy and Gran Sasso Science Institute, Viale F. Crispi, 7, 67100 L'Aquila, Italy (\email{piero.deidda@sns.it}).}
\and Martin Burger\thanks{ Helmholtz Imaging, Deutsches Elektronen Synchrotron (DESY) Notkestr. 85,
22607 Hamburg, Germany and Fachbereich Mathematik, Universität Hamburg,  Bundesstr. 55, 20146 Hamburg,
Germany (\email{martin.burger@desy.de})}
\and  Mario Putti\thanks{Department of Agronomy, Food, Natural Resources, Animals and Environment, University of Padova, Agripolis, Viale dell'Università 16, 35020 Legnaro (Padova) Italy} (\email{mario.putti@unipd.it})
\and Francesco Tudisco\thanks{School of Mathematics and Maxwell Institute for Mathematical Sciences, University of Edinburgh,  Peter Guthrie Tait Road, EH9 3FD, Edinburgh,
UK and Gran Sasso Science Institute, Viale F. Crispi, 7, 67100 L'Aquila, Italy (\email{f.tudisco@ed.ac.uk})} }

\ifpdf
\hypersetup{pdftitle={The graph infinity-Laplacian eigenvalue problem},
  pdfauthor={P. Deidda, M. Burger, M. Putti, and F. Tudisco}}
\fi

\begin{document}

\maketitle

\begin{abstract}
We analyze various formulations of the $\infty$-Laplacian eigenvalue problem on graphs, comparing their properties and highlighting their respective advantages and limitations. First, we investigate the graph $\infty$-eigenpairs arising as limits of $p$-Laplacian eigenpairs, extending key results from the continuous setting to the discrete domain. We prove that every limit of $p$-Laplacian eigenpair, for $p$ going to $\infty$, satisfies a limit eigenvalue equation and establish that the corresponding eigenvalue can be bounded from below by the packing radius of the graph, indexed by the number of nodal domains induced by the eigenfunction. Additionally, we show that the limits, for $p$ going to $\infty$, of the variational $p$-Laplacian eigenvalues are bounded both from above and from below by the packing radii, achieving equality for the smallest two variational eigenvalues and corresponding packing radii of the graph.
In the second part of the paper, we introduce generalized $\infty$-Laplacian eigenpairs as generalized critical points and values of the $\infty$-Rayleigh quotient. We prove that the generalized variational $\infty$-eigenvalues \new{equal the limit of the $p$-Laplacian variational eigenvalues and so} satisfy the same upper bounds in terms of packing radii. \new{Finally}, we establish that any solution to the limit eigenvalue equation is also a generalized eigenpair, while any generalized eigenpair satisfies the limit eigenvalue equation on a suitable subgraph. 
\end{abstract}



\section{Introduction}
We consider the problem of finding the eigenpairs of the $\infty$-Laplacian operator on graphs. While the continuous version of this problem has been studied extensively in the past few  decades~\cite{Lind2, Lind3, Lind4, Yu, Esposito, BungertInfNumerical}, its discrete counterpart on graphs remains largely unexplored. 
In recent years, eigenpairs of the graph $\infty$-Laplace operator have found potential application in a variety of fields, including machine learning and data analysis~\cite{Elmoataz2}, as well as $L^1$ optimal transport~\cite{bouchitte1997, Evans1999DifferentialEM} and the related problem of shape optimization~\cite{bouchitte2001}, spectral clustering and semi-supervised learning \cite{slepcev2019analysis, el2016asymptotic, calder2019consistency, kyng2015algorithms}, and image manipulation \cite{Elmoataz1, Elmoataz2}. 
\new{Both} in the continuous \cite{Kawohl2003, parini2010second, Lind2, Lind3} \new{and discrete} \cite{deidda2025nonlinear, Tudisco1, Bhuler, chang2016spectrum, zhang2021homological} settings, it is well known that the $p$-eigenpairs,  $p\in[1,\infty]$, encode topological information about the domain. Moreover, in~\cite{burger2016spectral, BungertNonlineardecomp}, it is shown that different variational filtering methods produce a decomposition of the signal in terms of nonlinear eigenfunctions of the filter. In particular, in the finite dimensional case any such suitable filter can be written as $\NORM[\infty]{Ax}$ for some proper matrix $A$. While this establishes a relationship with the discrete graph $\infty$-Laplacian eigenproblem, the study of these properties in the discrete setting is still largely unexplored. 
The primary challenge of the $\infty$-Laplacian eigenvalue problem arises from the non-differentiability of the $\infty$-norm, which necessitates either the extension of traditional methods to the non-smooth regime or the development of entirely new approaches.  When $1<p<\infty$, $p$-Laplacian eigenpairs are typically defined as critical points and values of the Rayleigh quotient $\rayl_p(\eigenfun) = \|\nabla \eigenfun\|_p/\|\eigenfun\|_p$. However, a straightforward generalization of the $\infty$-eigenpairs using an $\infty$-Rayleigh quotient is complicated by the lack of smoothness and continuity inherent in the $\infty$-norm.
Actually, we show in this work that different generalizations of the $p$-eigenpairs to the limit case $p=\infty$, lead to different notions of $\infty$-eigenpairs, each with distinct properties and implications.
For example, in the continuous setting, the study of the $\infty$-eigenpairs, started in~\cite{Lind2,Lind3}, is based on the study of solutions of the limit $p$-Laplacian eigenvalue equation when $p$ goes to $\infty$. However, in the discrete setting, the formulation based on the eigenvalue equation is known to be different from the one based on the idea of a generalized critical point theory for nonsmooth functionals employed in~\cite{hein2010inverse, chang2016spectrum} for the graph $1$-Laplacian eigenvalue problem.
Moreover, the generalized critical point theory has been successfully applied in recent works~\cite{BungertInfinityL2, bungert2021eigenvalue} to investigate minimizers of $\|\cdot\|_{\infty}$ in both the  $L^{\infty}$ and the $L^2$ settings, suggesting a promising direction for further explorations.
In this paper, we focus on the graph setting and examine the $\infty$-spectral problem from two complementary perspectives: the limit approach and the critical point theory. For clarity, we refer to the former as the $\infty$-limit eigenpairs and the latter as generalized $\infty$-eigenpairs.
Within the $\infty$-limit perspective, we address the extension to the discrete case of the results mainly obtained by Juutinen and co-authors \cite{Lind1, Lind2, Lind3, Lind4, Lind5}. More precisely, we are able to show that any $\infty$-limit eigenpair satisfies an $\infty$-limit eigenvalue equation. In addition, we prove that the $\infty$-limit of the $k$-th variational $p$-Laplacian eigenvalue can be bounded in terms of the $k$-th packing radius of the graph, i.e., the maximal radius that allows to inscribe $k$ distinct balls in the graph~\cite{grove1995new}. More general, we prove that the number of nodal domains of an $\infty$-eigenfunction can be used to bound the corresponding eigenvalue in terms of packing radii. Such results allow us to derive bounds from above and from below to the $\infty$-limit variational $p$-Laplacian eigenvalues in terms of the packing radii and, only for the first two $\infty$-limit variational eigenvalues, to establish equality with the first two packing radii. 
Concerning the second perspective, we study the generalized $\infty$-eigenpairs and the variational $\infty$-eigenvalue, which are a subset of the former. We provide a comparison between the two formulations of $\infty$-eigenpairs in the graph setting. We prove, first of all, that the variational generalized $\infty$-eigenvalues \new{are equal to the limit of the $p$-Laplacian variational eigenvalues}. Then, we show that the $\infty$-limit eigenvalue problem is ``stronger'' than the generalized eigenvalue problem in the sense of the following inclusions:
\begin{equation}\label{eq:eig_formulation_inclusion}
    \left\{\begin{array}{cc}
         &\text{Limit of }\\
         &\plap\text{-eigenpairs}
    \end{array}\right\}
    \subsetneq
    \left\{\begin{array}{cc}
         & \text{Solutions of the} \\
         & \text{limit eigenvalue eq.}
    \end{array} \right\} 
    \subsetneq 
    \left\{\begin{array}{cc}
         & \text{Generalized } \\
         & \infty\text{-eigenpairs }
    \end{array}\right \}\,.
\end{equation}
In more details, we consider the $\infty$-limit of $p$-Laplacian eigenpairs and prove that any accumulation point $(f,\Lambda)$ solves the system of equations
\begin{equation}\label{limit_eq}
  0=\begin{cases}
    \min\lbrace\NORM[\infty]{\grad f}-\Lambda f(u)\;,\;
    \inflap f(u)\rbrace \quad &\mathrm{if}\; f(u)>0\\
    \inflap f(u) \quad &\mathrm{if}\; f(u)=0\\
    \max\lbrace -\NORM[\infty]{\grad f}-\Lambda f(u) \;,\;
    \inflap f(u)\rbrace \quad &\mathrm{if}\; f(u)<0
  \end{cases}\,,
\end{equation}
but the inverse is not generally true.
We are also able to prove that, if $(f_k,\Lambda_k)$ is an accumulation point of the $p$-sequence of the $k$-th variational eigenpairs of the $p$-Laplacian, then
\begin{equation*}
  \Lambda_k\leq\frac{1}{R_k}\,,
\end{equation*}
where $R_k$ is the $k$-th (sphere) packing radius of the graph, i.e., the maximal radius that allows to inscribe in the graph $k$ distinct balls of radius $R_k$.
Moreover, denoting by $\mathcal{N}(f_k)$ the number of nodal domains induced by $f_k$, we can prove that 
\begin{equation*}
  \frac{1}{R_{\mathcal{N}(f_k)}}\leq \Lambda_k\,,
\end{equation*}
In particular, since we prove that for $k=1,2$ the number of nodal domains exactly equals the index, i.e. $\mathcal{N}(f_k)=k$, we are able to prove the following equalities:
\begin{equation*}
  \Lambda_k=\frac{1}{R_k}
  \quad \mbox{if} \quad
  k=1,2\,.
\end{equation*}
We obtain as a consequence that, as in the $p$-Laplacian case with $p\in(1,\infty)$ \cite{Tudisco1}, the number of nodal domains induced by the $\infty$-eigenfunctions can be used to establish relationships between the corresponding eigenvalues and the packing radii.

In~\cref{sec_inf_2} we consider the generalized eigenvalue problem 
\begin{equation}\label{gen_eq}
  0\in\partial\NORM[\infty]{\grad f}\cap \Lambda\partial\NORM[\infty]{f}\,,
\end{equation}
where $\partial\NORM[\infty]{\grad f}$ and $\partial\NORM[\infty]{f}$ denote the subdifferentials of the two associated functions of $f$. Using the generalized Lusternik-Schnirelman theory~\cite{chang2021nonsmooth}, we \new{observe that the Krasnoselskii variational $\infty$-eigenvalues}, that are always generalized eignvalues, correspond to the limit of the variational eigenvalues of the $p$-Laplacian. 
Finally, we compare the $\infty$-limit and the generalized formulations and, using a geometrical characterization of the eigenvalues and eigenfunctions, we prove that the former is stronger than the latter in the sense defined above in \eqref{eq:eig_formulation_inclusion}. In addition, we are able to prove that any solution of the second formulation, up to considering a subgraph, solves the $\infty$-limit equation \eqref{limit_eq}.

%


\section{Notation and Preliminaries}

In this section, we introduce our basic definitions and recall some preliminary results. In particular, since our goal is to prove discrete analogues of main results available for the continuous $\infty$-Laplacian eigenvalue problem, in this section, we also aim at discussing a solid framework for the discrete setting with definitions that find counterparts in the continuous case~\cite{Tudisco1, Bhuler, Amghibech1, Amghibech2}.  \new{In particular, we refer to \cite{manfredi2015nonlinear} for an overview about partial differential equations on graphs.}
\subsection{Graph setting and $p$-Laplacian operators}
We identify a graph $\Gc$ with a triplet $\Gc=(\nodeset,\edgeset,\edgelength)$, where $\nodeset$ denotes the set of nodes, $\edgeset\subset \nodeset\times\nodeset$ the set of edges, and $\edgelength:\edgeset\rightarrow\R$ is a weight on the edges. We identify an edge $(u,v)$ by the two nodes $u,v\in\nodeset$ that form its endpoints. In particular, the notation $u\sim v$ means that $(u,v)\in\edgeset$. A path $\Gamma$ connecting two nodes $v_0,v_n\in\nodeset$ on the graph is a set of edges that connect $v_0$ with $v_n$, i.e., $\Gamma=\{v_0\sim v_1\sim\dots\sim v_n\}$. We say that a graph is connected if given any two nodes $u,v\in\nodeset$, there exists a path that joins them. Since we are only interested in undirected graphs, throughout the paper we assume that whenever $(u,v)\in\edgeset$ then also $(v,u)\in\edgeset$ and $\edgelength(u,v)=\edgelength(v,u)$. To simplify the notation, in the following we write $\edgelength_{uv}:=\edgelength(u,v)$. Moreover, we assume that $\edgelength_{uv}$ can be interpreted as the inverse of the edge length leading to the following notion of edge length, and consequently, distance between two nodes.  
\begin{definition}[Length of a path and distance between two nodes]
   The length of a path $\Gamma$ is given by the sum of the reciprocals of the weights of the edges forming $\Gamma$,
  \begin{equation*}
    \length(\Gamma)=\sum_{i=1}^{n}\frac{1}{\edgelength_{v_{i-1}v_{i}}}\,.
  \end{equation*}
  The distance between two nodes $u,v\in\nodeset$ is defined as the length of the shortest path connecting $u$ and $v$,  \begin{equation}\label{distance_function}
    d(u,v)=\min\Big\lbrace \sum_{i=1}^n \edgelength_{v_{i-1} v_i}^{-1}\,: n\in\mathbb{N}, \; u=v_0\sim \dots \sim v_n=v\Big\rbrace\,.
  \end{equation}
\end{definition}
Throughout the manuscript, if not otherwise specified, we use capital letters (latin or greek) to denote edge functions and lowercase letters to denote node functions.

We proceed now by introducing discrete analogues of the continuous differential operators. The gradient of a function $f$ defined on the nodes, $f:\nodeset\rightarrow \R$, is the edge function $\grad f(u,v)$ given by
\begin{equation}\label{Grad-def}
  \grad f(u,v)=\edgelength_{uv}\big(f(v)-f(u)\big)\, ,
\end{equation}
for $(u,v)\in E$. 
Note that the linear operator $\grad$ is represented by an $|E|\times|V|$ matrix called the oriented incidence matrix of the graph.
Moreover, the assumption that the graph is undirected implies $\grad f(u,v)=-\grad f(v,u)$ for any $(u,v)\in\edgeset$.
The set of the gradients of $f$ outgoing from node $u$, is called the local gradient $\locgrad{f}(u)$ of $f$ at $u$:
\begin{equation*}
  \locgrad{f}(u):=\{\grad f (u,v)\;|\;v\sim u \}\,.
\end{equation*}
Note that, differently from the gradient operator in the continuous case, the cardinality of the local gradient is not fixed but depends on the degree of the node.
Next, given an edge function $G:\edgeset\rightarrow\R$, we introduce its divergence as the node function defined by the following expression:
\begin{equation}\label{Divergence-def}
  \Div{G}(u)=-\frac{1}{2}\grad^TG(u)
  =\frac{1}{2}\sum_{v\sim u}\edgelength_{uv}\left(G(u,v)-G(v,u)\right)
\end{equation}
Note that the above defined divergence operator satisfies the following ``integration by parts'' formula with respect to the scalar products $\SCAL[\edgeset]{\cdot}{\cdot}$ and $\SCAL[\nodeset]{\cdot}{\cdot}$  defined on the space of edge and node functions:
\begin{equation}\label{Integration_by_parts-divergence-theorem}
  \SCAL[\edgeset]{G}{\grad f}
  :=\frac{1}{2}\sum_{(u,v)\in\edgeset}G(u,v)\grad f(u,v)
  =\sum_{u\in\nodeset} -\Div{G}(u) f(u)
  =:\SCAL[\nodeset]{-\Div{G}}{f}\;.
\end{equation}
Observe that from the definition of the corresponding  scalar products defined in~\eqref{Integration_by_parts-divergence-theorem}, we can define the $p$-norms on the edge and node spaces as:
\begin{equation}\label{eq:norms}
    \NORM[p,\edgeset]{G}^p=\frac{1}{2}\sum_{(u,v)\in\edgeset}|G(u,v)|^p
    \qquad
    \NORM[p,\nodeset]{f}^p=\sum_{u\in\nodeset}|f(u)|^p\;,
\end{equation}
where we consider a one-half factor in edge-functions norms to take into account the fact that every edge has two possible directions and so it is counted twice.
Finally, borrowing from \cite{Tudisco1, DEIDDA2023_nod_dom, Bhuler}, we define the $p$-Laplace operator as follows:
\begin{definition}[Discrete $p$-Laplace Operator] \label{discrete_plap-def}
  The discrete $p$-Laplace operator is given by:
  \begin{equation*}
    \plap f(u):=
    -\divg(|\grad f|^{p-2}\odot\grad f)
    =\sum_{v\sim u}\edgelength_{uv}^{p}|f(u)-f(v)|^{p-2}\left(f(u)-f(v)\right)\,,
  \end{equation*}
  where $v\sim u$ means that $(u,v)\in \edgeset$, the absolute value in $|\grad f|$ is considered entry-wise, and ``\,$\odot$'' denotes the Hadamard or entry-wise product.
\end{definition}
Next, we introduce the discrete counterpart of the homogeneous Dirichlet boundary conditions. To this aim, we designate an arbitrary subset of the nodes $\boundary\subset\nodeset$ as the boundary of $\Gc$ and consider the set of node functions that are zero on $\boundary$: $\mathcal{H}_0(\nodeset):=\{f\,|\;f(u)=0 \;\forall u\in \boundary\}$.
From the study of the critical points of the $p$-Rayleigh quotient $\rayl_{p}:\Hc_0(\nodeset)\rightarrow\R$ given by
\begin{equation}\label{p-Rayleigh-quotient}
  \rayl_p(f)=\frac{\NORM[p]{\grad f}}{\NORM[p]{f}}
\end{equation}
we can write the discrete equivalent of the $p$-Laplacian eigenvalue equation with homogeneous Dirichlet boundary conditions \cite{Hua, PARK2011, Park, manfredi2015nonlinear} as:
\begin{equation}\label{Dirichlet-p_lap_eigenvalue_eq-intro}
  \begin{cases}
    \plap f(u)=\lambda |f|^{p-2}(u)f(u)\quad &u\in\internalnodes\\
    f(u)=0 & u\in\boundary
  \end{cases}\,.
\end{equation}
If $\boundary\neq\emptyset$, we can introduce the distance from the boundary given by:
\begin{equation}\label{boundary_distance}
  d_{\boundary}(u)=\min_{v\in\boundary}d(u,v)\,.
\end{equation}
In particular, if a graph has nonempty boundary we denote by $\internalnodes:=\nodeset\setminus\boundary$ the set of internal nodes and we say that the graph is connected if the subgraph induced by the internal nodes is connected. If not otherwise stated and independently of the presence of the boundary, throughout the paper we assume a graph to be connected.
We define the $p$-Laplacian variational eigenvalues $\lambda_k(\plap)$ as in~\cite{Tudisco1, struwe} and write:
\begin{equation}\label{DEf:p-Lap_var_eigenvalues}
  \lambda_k(\plap)=
  \min_{A\in\mathcal{F}_k(S_p)}\max_{f\in A}\rayl_p^p(f)\,,
\end{equation}
where $S_p:=\{f\in \Hc_0(\nodeset)|\; \|f\|_p=1\}$ and $\mathcal{F}_k(S_p)$ is the set of the closed and symmetric subsets of $S_p$ with Krasnoselskii genus greater than or equal to $k$. We recall that the Krasnoselskii genus of a closed symmetric set $A$ is defined as follows~\cite{struwe, Ghoussoub, Fucik}:
\begin{definition}[Krasnoselskii Genus]\label{Def-Krasnoselskii-genus}
  The Krasnoselskii Genus $\gamma(A)$ of a closed symmetric set $A$ is:
  \begin{equation}\label{Krasnoselskii_genus_Infinity}
    \gamma(A):=
    \begin{cases}
      0 & \text{if $A=\emptyset$}
      \\
      \inf\left\{h\in\mathbb{N}\,: \,
      \exists\,\varphi\in C(A, \R^h\setminus\{0\})
      \text{ s.t. } \varphi(x)=-\varphi(-x)\right\} &
      \\
      \infty \quad & \text{if $\nexists \; h$ as above}
    \end{cases}.
  \end{equation}
\end{definition}
Given the $\infty$-norms of a node function $f$ and of its gradient:
\begin{equation}\label{Infinity-Norms}
  \NORM[\infty]{f}=\max_{u\in\nodeset}\{|f(u)|\}\,,\qquad
  \NORM[\infty]{\grad f}=\max_{(u,v)\in\edgeset}\{|\grad f(u,v)|\},
\end{equation}
we can introduce the maximal node and edge sets:
\begin{equation*}
\begin{aligned}
  \nodemaxset(f)&=\{u\in\nodeset\::\: |f(u)|=\NORM[\infty]{f}\}\,,\\
  \edgemaxset(f)&=\{(u,v)\in\edgeset\::\: |\grad f(u,v)|=\NORM[\infty]{\grad f}\}\, .
\end{aligned}
\end{equation*}
Before considering the $\infty$-eigenvalue problem we need to introduce the corresponding $\infty$-Laplacian operator. We start by recalling the definition in the continuous setting~\cite{Lind1} (differently from ~\cite{Lind1} we are adding a minus sign to be consistent with our definition of the $p$-Laplacian operator that is positive definite differently from the one considered in \cite{Lind1}):
\begin{definition}[Continuous $\infty$-Laplacian]
  \label{Continuous_Inf_lap_def}
  \begin{equation*}
    \inflap f:=-\sum_{i,j}^n
    \frac{\partial f}{\partial x_i}
    \frac{\partial f}{\partial x_j}
    \frac{\partial^2 f}{\partial x_i\partial x_j}
    =-(\nabla f)^T H(f) \nabla f \,,
  \end{equation*}
  where $H(f)$ denotes the Hessian matrix of a twice differentiable function $f$ and $\nabla f$ its standard gradient.
\end{definition}
Clearly, differently from the $p$-Laplacian operator with $p<\infty$, this definition cannot be extended to the discrete setting in a straight-forward manner because of the presence of the Hessian matrix. Nevertheless, a definition of the discrete $\infty$-Laplacian operator is given in~\cite{manfredi2015nonlinear, Elmoataz1, elmoataz2011infinity} as follows:
\begin{definition}[Discrete $\infty$-Laplacian Operator]
  \label{discrete_inf_lap_def}
 The discrete $\infty$-Laplacian operator can be written as:
  \begin{equation*}
    \inflap f(u)=\Big(\| \Opc{\grad f}^{-}(u) \|_{\infty}-\| \Opc{\grad f}^{+}(u) \|_{\infty}\Big)\,,
  \end{equation*}
  where $x^+:=\max\lbrace x,0 \rbrace$, $x^-:=\max\lbrace -x,0 \rbrace$. 
\end{definition}
This definition is consistent with the continous setting. Indeed, indicating with $f$ the limit as $p$ goes to $\infty$ of the solutions of $\plap f_p=0$ with prescribed values on the boundary, it is easy to see that $f$ must satisfy $\inflap f=0$ for all $u\in\internalnodes$~\cite{Elmoataz2, Elmoataz1, elmoataz2011infinity, Lind1, manfredi2015nonlinear}. 
In particular, the above definition is a direct consequence of the following reformulation of the discrete $p$-Laplacian operator \cite{elmoataz2011infinity}:
\begin{equation}\label{discrete_plap_pnorms-difference_def}
 \begin{aligned}
   \plap f(u):
   =&\sum_{v\in \nodeset}
      \edgelength_{uv}^{p}|f(u)-f(v)|^{p-2}\big(f(u)-f(v)\big)\\
   =&\Big(\sum_{v\in\nodeset}\edgelength_{uv}^{p}
      \big|\big(f(u)-f(v)\big)^{+}\big|^{p-1}-
      \sum_{v\in\nodeset} \edgelength_{uv}^{p}
      \big|\big(f(u)-f(v)\big)^-\big|^{p-1}\Big)\\
   =&\Big(\| \Opc{\grad[\edgelength'] f}^{-}(u)
      \|_{{p-1}}^{p-1}-\| \Opc{\grad[\edgelength'] f}^{+}(u) \|_{{p-1}}^{p-1}\Big)\,,
 \end{aligned}
\end{equation}
where 
$\edgelength'_{uv}=\edgelength_{uv}^{\frac{p}{p-1}}$.

\begin{remark}\label{Remark_boundary_conditions}
  Note that, given a graph with boundary $\boundary$, using the inclusion operator $P:\mathcal{H}(\internalnodes)\rightarrow\mathcal{H}_0(\nodeset)$, it is easy to see that the space $\mathcal{H}_0(\nodeset)$ is isomorphic to the space $\mathcal{H}(\internalnodes):=\{f:\internalnodes\rightarrow\R\}$: 
  \begin{equation*}
    Pf(u)=\begin{cases}
      f(u) \quad\forall u\in\internalnodes\\
      0 \quad\forall u\in\boundary
    \end{cases}.
  \end{equation*}
  This implies that $Pf$ is a $p$-Laplacian eigenfunction if $f$ satisfies: 
  \begin{equation*}
    \frac{1}{2}P^T \grad^T(|\grad P f|^{p-2}\odot \grad P f)(u)
    =\lambda |f(u)|^{p-2} f(u) \qquad \forall u\in\internalnodes,
  \end{equation*}
  allowing us to forget the homogeneous Dirichlet boundary conditions. Thus, for the sake of notation, we will often refer to the eigenvalue problem in $\mathcal{H}(\internalnodes)$ and, with a small abuse of notation we will write $\grad: \mathcal{H}(\internalnodes)\rightarrow\mathcal{H}(\edgeset)$, $\divg:\mathcal{H}(\edgeset) \rightarrow \mathcal{H}(\internalnodes)$ and $\plap: \mathcal{H}(\internalnodes) \rightarrow \mathcal{H}(\internalnodes)$ to denote 
  \begin{equation*}
    \grad:=\grad P, \qquad \divg:=P^T \divg,
    \qquad
    \plap f:= - P^T \divg (|\grad P f|^{p-2}\odot \grad P f)\,,
  \end{equation*}
  where, in the last equalities, the $\grad$ and $\divg$ operators appearing in right hand sides are defined according to the definitions given in eqs. \eqref{Grad-def} and \eqref{Divergence-def}. Similarly, given a function $f\in\mathcal{H}(\internalnodes)$, we can modify the definition of the local gradient, of the $\infty$-Laplacian, and of the $p$-norm in accordance with the definition of the $p$-Rayleigh quotient:
  \begin{equation*}
    \begin{aligned}
      & \grad f(u):=\{\grad Pf(u,v)|v\sim u\}\qquad u\in\internalnodes \\
      & \inflap f(u)=\Big(\|\Opc{\grad P f}^{-}(u)\|_{\infty}- \|\Opc{\grad P f}^{+}(u)\|_{\infty}\Big) \qquad u\in\internalnodes
    \end{aligned}
  \end{equation*}
\begin{equation*}
  \NORM[p,\nodeset]{f}^p=
  \sum_{u\in\nodeset}|Pf(u)|^p
  =\sum_{u\in\internalnodes}|f(u)|^p
  \qquad
  \NORM[\infty]{f}=\max_{u\in\internalnodes}\{|f(u)|\}.
\end{equation*}
\end{remark}

%


\section{The $\infty$-limit eigenvalue problem}\label{sec_inf_1}
In this section we start discussing the $\infty$-Laplacian eigenvalue problem. In particular we concentrate on the ``$\infty$-limit eigenvalue problem'', i.e. we consider the limit as $p\To\infty$ of the $p$-Laplacian eigenvalue equation and we discuss its solutions. This approach has been thoroughly investigated in the continuous setting \cite{Lind2,Lind3,Lind4, Lind5, Esposito, Yu}. Here, we start by recalling the main results in the continuous setting and then go on to prove analogous results in the discrete setting. We point out that, even if the results in the two settings are equivalent, the approaches to obtain them are different.
We start by observing that, in both continuous and discrete settings, the limit of appropriate subsequences of $p$-Laplacian eigenpairs $(f(\plap),\lambda(\plap))$ leads to an $\infty$-eigenfunction whose corresponding eigenvalue can be expressed formally as:
\begin{equation}\label{inf_eigenvalues_def}
\Lambda=``\lim_{p\rightarrow \infty}\big(\lambda(\plap)\big)^{\frac{1}{p}}\,"\,.
\end{equation}
\new{In particular, in the discrete setting} we know that the set $\lbrace\lambda^{\frac{1}{p}}(\plap)\rbrace_p$ is bounded \cite{DEIDDA2023_nod_dom} and hence we can extract from it at least one convergent subsequence. 

\begin{NEW}
In particular, as a consequence of Lemma 4.2 in~\cite{deidda2025nonlinear}, for $1\leq  p\leq q$ and any $f\in \mathcal{H}_0(\nodeset)$ it holds 
\begin{equation}\label{eq_ineq_rayl_pq}
\left(|E|/2\right)^{\frac{1}{p}-\frac{1}{q}}\rayl_p(f)\leq \rayl_q(f)\leq |\internalnodes|^{\frac{1}{p}-\frac{1}{q}}\rayl_p(f),
\end{equation}
where $|\internalnodes|$ and $|\edgeset|$ correspond to the cardinality of the sets of internal nodes and edges, respectively.
    Analogously to \eqref{DEf:p-Lap_var_eigenvalues}, we define
    \begin{equation}\label{def_inf_var_eigenvalues}
      \Lambda_k(\inflap):=
      \min_{A\in\mathcal{F}_k(S_\infty)}\max_{f\in A}\rayl_\infty(f)\,,
       \end{equation}
     where $\mathcal{F}_k(S_\infty)$ is the set of the closed and symmetric subsets of $S_\infty$ with Krasnoselskii genus greater than or equal to $k$.
     Then as a consequence of \eqref{eq_ineq_rayl_pq} we have    %
\begin{equation}\label{eq_monotonocity_variational_eigenvalue}
    \left(|E|/2\right)^{\frac{1}{q}-\frac{1}{p}}\lambda_k(\Delta_p)^{\frac{1}{p}}\leq \lambda_k(\Delta_q)^{\frac{1}{q}}\leq |\internalnodes|^{\frac{1}{p}-\frac{1}{q}}\lambda_k(\Delta_p)^{\frac{1}{p}},
\end{equation}
which holds also for $q=\infty$ by replacing $\lambda_k(\Delta_q)^{1/q}$ with $\Lambda_k(\Delta_\infty)$.
 The last equation proves that the limit of the variational eigenvalues exists and is given by: 
\begin{equation}
    \lim_{p\to \infty}\lambda_k(\plap)^{1/p}=\Lambda_k(\inflap).    
\end{equation}
Additionally, eq. \eqref{eq_monotonocity_variational_eigenvalue} shows that
 $$|\lambda_k(\plap)^{1/p}-\Lambda_k(\Delta_\infty)|\leq C\max\{(1-\left(|E|/2\right)^{-\frac{1}{p}}),(|\internalnodes|^{\frac{1}{p}}-1)\},$$
 where $C$ bounds $\lambda_k(\plap)^\frac{1}{p}$. Note that the last inequality also states that the variational eigenvalues converge at least as $O(1/p)$.
\end{NEW}
In the continuous setting, the limit of the $p$-eigenpairs can be studied by considering the $\infty$-eigenvalue equation for a general eigenvalue $\Lambda$, see~\cite{Lind3}:
\begin{equation}\label{inf_continuous_eigenvalue_eq}
  0=\begin{cases}
    \min\lbrace |\grad f|-\Lambda f,\inflap f \rbrace \quad &\mathrm{if}\; f>0\\
    \inflap f   &\mathrm{if}\; f=0\\
    \max\lbrace -|\grad f|-\Lambda f,\inflap f \rbrace \quad &\mathrm{if}\; f<0
  \end{cases}\,,
\end{equation}
where the above equation has to be understood in the sense of viscosity solutions (see \cite{Lind2} for a short and clear discussion about viscosity solutions in this setting). 
In particular it is possible to show that if $(f,\Lambda)=\lim_{p\rightarrow \infty}(f_p,\lambda_p^{\frac{1}{p}})$, where $(f_p,\lambda_p)$ is an eigenpair of the $p$-Laplacian for a given $p$ then $(f,\Lambda)$ solves \eqref{inf_continuous_eigenvalue_eq}. In addition, whenever a pair $(f,\Lambda)$ solves eq.\eqref{inf_continuous_eigenvalue_eq}, then
\begin{equation*}
  \Lambda=\frac{\| \grad f\|_{\infty}}{\| f \|_{\infty}}\;.
\end{equation*}
By means of this equation it is possible to relate the limit of the variational $p$-Laplacian eigenvalues to the packing radii of $\Omega$. Indeed, given a bounded domain $\Omega\in\R^n$ and an integer $k$, the $k$-th packing radius of $\Omega$ is defined as the maximal radius $r\in\R$ that allows inscribing $k$ distinct balls in $\Omega$, i.e.:
\begin{equation*}
  R_k:=\sup\lbrace r\in\R \text{ s.t. }
  \exists B_r(x_1),\dots,B_{r}(x_k)\subset \Omega,
  \; B_r(x_i)\cap B_r(x_j)=\emptyset\;
  \forall i,j=1,\dots,k \rbrace\,,
\end{equation*}
where $B_r(x_i)$ is the $n$-dimensional ball of radius $r$ centered in $x_i$, $x_i\in\Omega$. Then, given $\Lambda_k$ an accumulation point of the sequence of the $k$-th variational eigenvalues of the \new{continuous}  $p$-Laplacian, i.e.
\begin{equation*}
  \Lambda_k=\lim_{h\rightarrow \infty} \lambda_k^{\frac{1}{p_h}}(\Delta_{p_h})\,,
\end{equation*}
it is possible to relate such value to $R_k$ (see \cite{Lind3,Lind2}) and obtain the following inequality: 
\begin{equation*}
  \Lambda_k\leq \frac{1}{R_k}\,.
\end{equation*}
For the particular cases $k=1,2$, also the opposite inequality holds and the sequences $\{\lambda_{1,2}^{\frac{1}{p}}(\plap)\}_{p=2}^{\infty}$ converge. Thus, we can write:
\begin{equation*}
  \Lambda_1=\frac{1}{\|d_{\partial\Omega}\|_{\infty}}\,, \qquad \Lambda_2= \frac{1}{R_2}\,,
\end{equation*}
where $\partial\Omega$ denotes the boundary of the domain $\Omega$. Several studies have concerned the $\infty$-eigenfunctions associated to the first eigenvalue $\Lambda_1$ (see \cite{Lind1,Yu, juutinen1999infinity, hynd2013nonuniqueness}).
Indeed, despite the fact that the boundary distance function $d_{\partial\Omega}$ is always a minimizer of $\rayl_{\infty}$, it is not always an eigenfunction, meaning that, depending on $\Omega$, the pair $(d_{\partial\Omega},1/\|d_{\partial\Omega}\|_{\infty})$ may not solve eq.~\eqref{inf_continuous_eigenvalue_eq}.
In addition, the first eigenfunction may not be unique and there may be first eigenfunctions that solve the $\infty$-limit eigenvalue equation without beeing limits of eigenfunctions of the $p$-Laplacian.



\subsection{The $\infty$-limit eigenvalue equation on a graph}
In this section we formulate discrete analogues of the $\infty$-eigenvalue eq.~\eqref{inf_continuous_eigenvalue_eq} and discuss their solutions. Note that, given a a function $f\in\mathcal{H}(\internalnodes)$ , using the equalities in~\eqref{discrete_plap_pnorms-difference_def} we can rewrite the eigenvalue equation of the $p$-Laplacian given in \eqref{Dirichlet-p_lap_eigenvalue_eq-intro} as:
\begin{subequations}
  \begin{align}
    \label{plap_splitted_eigenvalue_eq}
    &\begin{cases}
      \left(\NORM{\Opc{\grad[\edgelength'] f}^-(u)}_{p-1}^{p-1}
      -\NORM{\Opc{\grad[\edgelength'] f}^+(u)}_{p-1}^{p-1}\right)^{\frac{1}{p-1}}
      =\lambda^{\frac{1}{p-1}} f(u) \\[1.1ex]
      \NORM{\Opc{\grad[\edgelength'] f}^-(u)}_{p-1}
      -\NORM{\Opc{\grad[\edgelength'] f}^+(u)}_{p-1}>0 
    \end{cases} \!\!\!\!\!\!\!\!
      \quad\text{if } f(u)>0 \\[1.5em]
    &\begin{cases}
      \left(\NORM{\Opc{\grad[\edgelength'] f}^+(u)}_{p-1}^{p-1}
      -\NORM{\Opc{\grad[\edgelength']f}^-(u)}_{p-1}^{p-1}\right)^{\frac{1}{p-1}}
      =-\lambda^{\frac{1}{p-1}} f(u) \\[1.1ex]
      \NORM{\Opc{\grad[\edgelength'] f}^-(u)}_{p-1}
      -\NORM{\Opc{\grad[\edgelength'] f}^+(u)}_{p-1}^{p-1}<0 
    \end{cases}\!\!\!\!\!\!\!\!
      \quad\text{if } f(u)<0 \\[1.5em]
    &\phantom{.}\,\:\NORM{\Opc{\grad[\edgelength'] f}^-(u)}_{p-1}
      -\NORM{\Opc{\grad[\edgelength'] f}^+(u)}_{p-1}=0\,\,
      \qquad\text{if } f(u)=0
  \end{align}
\end{subequations}
where, as in eq.~\eqref{discrete_plap_pnorms-difference_def}, $\edgelength'_{uv}:=\edgelength_{uv}^{\frac{p}{p-1}}$.
In the following Theorem we prove that any accumulation point of a sequence of $p$-Laplacian eigenpairs satisfies the $\infty$-limit eigenvalue equation. 
Before stating the Theorem, we would like to observe that, as shown in~\cite{DEIDDA2023_nod_dom}, any $p$-Laplacian eigenvalue $\lambda(\plap)$ satisfies the following inequality:
\begin{equation*}
  \lambda(\plap)\leq
  \max_{u\in\internalnodes} 2^{p-1}\sum_{v\sim u}\edgelength_{uv}^p  \,.
\end{equation*}
Thus, the set of the $p$-roots of the $p$-Laplacian eigenvalues as $p$ varies is bounded, i.e.:
\begin{equation*}
  \text{there exists a constant } C>0
  \text{ s.t. } \lambda^{\frac{1}{p}}<C
  \text{ for all } \lambda\in\sigma(\plap)\,,
\end{equation*}
where $\sigma(\plap)$ denotes the spectrum of the graph $p$-Laplacian.
\begin{theorem}\label{limitin_infty_eigenvalue_theorem}
  Let $(f_{p_j},\lambda_{{p_j}})$ be a sequence of $p$-Laplacian
  eigenpairs and assume that $\displaystyle\lim_{j\rightarrow
    \infty}(f_{p_j},\lambda_{p_j}^{\frac{1}{p_j}})=(f,\Lambda)$. Then
  the pair $(f,\Lambda)$ satisfies the following set of equations \new{for all $u\in\internalnodes$}:
  \begin{equation}\label{limitin_inf_eigenvalue_eq}
    0=\new{\Theta_{\infty}(f)(u)=}\begin{cases}
      \min\lbrace\|\Opc{\grad f}(u)\|_{\infty}- \Lambda f(u)\,,\, \inflap f(u)\rbrace \quad &\text{if } f(u)>0\\
      \inflap f(u)
      \quad &\text{if } f(u)=0\\
      \max\lbrace -\|\Opc{\grad f}(u)\|_{\infty}- \Lambda f(u) \,,\, \inflap f(u)\rbrace \quad &\text{if } f(u)<0
    \end{cases}.
  \end{equation}
\end{theorem}
\begin{proof}
  To prove the assertion, consider the case $f(u)>0$. Then, for any $j$ large enough,  by \eqref{plap_splitted_eigenvalue_eq}, we can write
  \begin{equation}\label{eq1_limit_eig_eq}
    \begin{cases}
    \displaystyle{\| \Opc{\grad[\edgelength'] f_{p_j}}^-(u) \|_{p_j-1}\Big(1-\Big(\frac{\| \Opc{\grad[\edgelength'] f_{p_j}}^+(u) \|_{p_j-1}}{\| \Opc{\grad[\edgelength'] f_{p_j}}^-(u) \|_{p_j-1}}\Big)^{p_j-1}\Big)^{\frac{1}{p_j-1}} =\lambda_{p_j}^{\frac{1}{p_j-1}} f_{p_j}(u)\,,} \\
      \| \Opc{\grad[\edgelength'] f_{p_j}}^-(u) \|_{p_j-1} >
      \| \Opc{\grad[\edgelength'] f_{p_j}}^+(u) \|_{p_j-1}\,.
    \end{cases}
  \end{equation}
  Taking the limit we have
  \begin{equation}\label{eq2_limit_eig_eq}
    \begin{cases}\displaystyle
      \|\Opc{\grad f}(u)
      \|_{\infty}
      \lim_{j\rightarrow\infty}
      \Big(1-\Big(
      \frac{\|\Opc{\grad[\edgelength']f_{p_j}}^+(u) \|_{p_j-1}}%
      {\|\Opc{\grad[\edgelength']f_{p_j}}^-(u)\|_{p_j-1}}
      \Big)^{p_j-1}\Big)^{\frac{1}{p_j-1}}
      =\Lambda f(u) \\ 
      \|\Opc{\grad f}^-(u)\|_{\infty}\geq\|\Opc{\grad f}^+(u)\|_{\infty}
    \end{cases},
  \end{equation}
  where, we have used the equality $\|\Opc{\grad f}^-(u)\|_{\infty}=\|\Opc{\grad f}(u)\|_{\infty}$ that follows trivially from the second inequality in~\eqref{eq1_limit_eig_eq}. From this same inequality, we observe that 
  \begin{equation*}
    \lim_{j\rightarrow\infty}\bigg(1-\bigg(\frac{\| \Opc{\grad[\edgelength'] f_{p_j}}^+(u) \|_{p_j-1}}{\| \Opc{\grad[\edgelength'] f_{p_j}}^-(u) \|_{p_j-1}}\bigg)^{p_j-1}\bigg)^{\frac{1}{p_j-1}}\leq 1,
  \end{equation*}
and thus, if $\NORM[\infty]{\Opc{\grad f}^-(u)} \gneqq \| \Opc{\grad f}^+(u) \|_{\infty}$, we obtain:
  \begin{equation*}
    \lim_{j\rightarrow\infty}\bigg(1-\bigg(\frac{\| \Opc{\grad[\edgelength'] f_{p_j}}^+(u) \|_{p_j-1}}{\| \Opc{\grad[\edgelength'] f_{p_j}}^-(u) \|_{p_j-1}}\bigg)^{p_j-1}\bigg)^{\frac{1}{p_j-1}}= 1.
  \end{equation*}
  Replacing the last equation in~\eqref{eq2_limit_eig_eq} we obtain:
  \begin{equation*}
    \| \Opc{\grad f}(u) \|_{\infty} =\Lambda f(u)\,.
  \end{equation*}
  Vice versa, if  ${\displaystyle\lim_{j\rightarrow\infty}}\Big(1-\Big(\frac{\| \Opc{\grad[\edgelength'] f_{p_j}}^+(u) \|_{p_j-1}}{\| \Opc{\grad[\edgelength'] f_{p_j}}^-(u) \|_{p_j-1}}\Big)^{p_j-1}\Big)^{\frac{1}{p_j-1}}< 1$, then necessarily
  we have 
  \begin{equation*}
    \| \Opc{\grad f}^-(u) \|_{\infty} = \| \Opc{\grad f}^+(u) \|_{\infty}\,,
  \end{equation*}
  showing that~\eqref{limitin_inf_eigenvalue_eq} is satisfied.
  The cases $f(u)=0$ and $f(u)<0$ can be proved analogously, concluding the proof.
\end{proof}
We point out that eq.~\eqref{limitin_inf_eigenvalue_eq} is very similar to the continuous analogue in~\eqref{inf_continuous_eigenvalue_eq}, with the only difference being the absolute value replaced by the $\infty$-norm of the local gradient.
Next we show that whenever $(f,\Lambda)$ satisfies the limit eigenvalue eq~.\eqref{limitin_inf_eigenvalue_eq}, $\Lambda$ matches the value of the $\infty$-Rayleigh quotient computed in $f$, i.e. $\rayl_{\infty}(f)$. 
\begin{proposition}\label{Prop_maximal_points}
  Assume that $(f,\Lambda)$ satifies eq.~\eqref{limitin_inf_eigenvalue_eq}, $f$ is not a constant function, and $u\in\nodemaxset(f)$. Then $\|\grad f\|_{\infty}=\| \Opc{\grad f}(u) \|_{\infty}=\Lambda |f(u)|=\Lambda \| f \|_{\infty}$, i.e. 
  \begin{equation*}
    \Lambda=\frac{\| \grad f \|_{\infty}}{\| f \|_{\infty}}\,.
  \end{equation*}
\end{proposition}
\begin{proof}
  Assume $u\in\internalnodes$ to be a node such that $|f(u)|=\| f \|_{\infty}$, i.e., $u\in\nodemaxset(f)$, and, without loss of generality, that $f(u)>0$. Then~\eqref{limitin_inf_eigenvalue_eq} ensures that:  
  \begin{equation*}
    \|\Opc{\grad f}(u)\|_{\infty}\geq \Lambda f(u) >0\,.
  \end{equation*}
  Since $u$ is a maximum point, we have that $\|\Opc{\grad f}^{+}(u)\|_{\infty}=0$, which, together with the previous inequality, yields:
  \begin{equation*}
    \inflap f(u)>0 \,.
  \end{equation*}
  Thus, from \eqref{limitin_inf_eigenvalue_eq} we necessarily have:
  \begin{equation*}
    \|\Opc{\grad f}(u)\|_{\infty}-\Lambda f(u)=0\,.
  \end{equation*}
  The last equality allows us to observe that, if the thesis holds for some $w\in \nodemaxset(f)$, it holds as well for all the other $u\in\nodemaxset(f)$. Indeed
  \begin{equation*}
    \frac{\|\Opc{\grad f}(u)\|_{\infty}}{|f(u)|}=\frac{\|\Opc{\grad f}(u)\|_{\infty}}{\|f\|_{\infty}}=\Lambda=\frac{\|\Opc{\grad f}(w)\|_{\infty}}{|f(w)|}=\frac{\|\grad f\|_{\infty}}{\|f\|_{\infty}}\,.
  \end{equation*}
  To prove the existence of such a node $w$, we proceed as follows. Let $u_0\in\internalnodes$ be a node admitting an edge $(v_0,u_0)$ with $|\grad  f(v_0, u_0)|=\| \grad f \|_{\infty}=\|\Opc{\grad  f}(u_0)\|_{\infty}=\|\Opc{\grad  f}(v_0)\|_{\infty}$. Then, if $\inflap f(u_0)\neq 0$, we can write:
  \begin{equation*}
    \Lambda=\frac{\| \Opc{\grad f}(u_0) \|_{\infty}}{|f(u_0)|}=\frac{\| \grad f \|_{\infty}}{|f(u_0)|}\geq \frac{\| \Opc{\grad f}(u) \|_{\infty}}{|f(u)|}=\frac{\| \Opc{\grad f}(u) \|_{\infty}}{\| f \|_{\infty}}=\Lambda\,,
  \end{equation*}
  where $u$ is any node satisfying $|f(u)|=\|f\|_{\infty}$, yielding in particular $|f(u_0)|=\| f \|_{\infty}$ and thus the thesis.
  On the other hand, if $\inflap f(u_0)=\inflap f(v_0)=0$, assuming without loss of generality $f(u_0)>f(v_0)$ (indeed the role of $u_0$ and $v_0$ is interchangeable), there exists $u_1\sim u_0$ such that $f(u_1)>f(u_0)$ with
  \begin{equation*}
    \begin{aligned}
      \edgelength_{u_0 u_1}\big(f(u_1)-f(u_0)\big)&=\|\Opc{\grad f}^-(u_0)\|_{\infty}\\
                                                  &=\|\Opc{\grad f}^+(u_0)\|_{\infty}=\edgelength_{u_0,v_0}\big(f(u_0)-f(v_0)\big)=\| \grad f \|_{\infty}\,.
    \end{aligned}
  \end{equation*}
  Thus there exists another edge $(u_0,u_1)$ such that $|\grad f (u_0,u_1)|=\| \grad f \|_{\infty}$ and $f(u_1)>f(u_0)>f(v_0)$. Iterating this procedure, by the finiteness of the graph and the previous argument, there must exist a node $u_k$ such that:
  \begin{equation*}
    \| \Opc{\grad f}(u_k) \|_{\infty}=\| \grad f \|_{\infty}=\Lambda |f(u_k)|=\Lambda \| f \|_{\infty}\,.
  \end{equation*}
\end{proof}


\subsection{Topological Properties of the $\infty$-eigenpairs}
This subsection studies some relationships between the solutions of $\eqref{limitin_inf_eigenvalue_eq}$, the limits of $p$-Laplacian eigenpairs, and the topology of the graph.
In particular, we aim at recovering a discrete version of the relationships holding in the continuous setting between the limit variational eigenvalues and the packing radii of the domain. 
Within this framework, the distance function introduced in~\eqref{distance_function} and its boundary counterpart in~\eqref{boundary_distance} become significant tools in the study of the $\infty$-eigenpairs interpreted as solutions of~\eqref{limitin_inf_eigenvalue_eq}.   
\begin{proposition}\label{optimal_paths_limiting_eigenfunctions}
  Let $(f,\Lambda)$, be a pair satisfying~\eqref{limitin_inf_eigenvalue_eq} with $f$ not a constant function. Then for any $u\in\nodemaxset(f)$ there exists a path $\Gamma=\{(u_i,u_{i+1})\}_{i=1}^{n-1}$ starting from $u$ ($u_1=u$) such that :
  \begin{enumerate}
  \item $f$ is monotone along $\Gamma$ and  for all $1\le k\le n-1$, $\grad f(u_k,u_{k+1})=\|\grad f\|_{\infty}$;\\[-.8ex]
  \item\label{enum:cases}
    either
    $
      \begin{cases}
        u_n\in\boundary \\
        \displaystyle{\Lambda=\frac{1}{\length(\Gamma)}}
      \end{cases} 
    $
    or 
    $
      \begin{cases}
        f(u_n)=-f(u) \\
        \displaystyle{\Lambda=\frac{2}{\length(\Gamma)}}
      \end{cases}$;\\[.9ex]
  \item $\displaystyle{\Lambda=\min\bigg\{\frac{1}{d_{\boundary}(u)}, \min_{\{v\,|\,f(v)=-f(u)\}}\frac{2}{d(u,v)} \bigg\}}\,.$
  \end{enumerate}
\end{proposition}
\begin{proof}
  The first two points follow from Proposition~\ref{Prop_maximal_points} and eq.~\eqref{limitin_inf_eigenvalue_eq}.
  Indeed, given $u_1=u$, by Proposition~\ref{Prop_maximal_points} there exists $u_2\sim u_1$ such that $|\grad f(u_1,u_2)|=\|\grad f\|_{\infty}$.
  Then, from \eqref{limitin_inf_eigenvalue_eq} we conclude that one of the three following conditions must hold:
  \begin{equation*}
    f(u_2)=-f(u_1)  \qquad\mbox{or}\qquad
    u_2\in\boundary \qquad\mbox{or}\qquad
    |f(u_2)|<|f(u_1)|\,.
  \end{equation*}
  In the last case, since $\|\grad f(u_2)\|_{\infty}=\|\grad f\|_{\infty}$, necessarily $\inflap f(u_2)=0$, i.e. there exists $u_3 \sim u_2$ such that $|\grad f(u_2,u_3)|=\|\grad f\|_{\infty}$ and, assuming w.l.o.g. $f(u_1)>f(u_2)$, it follows $f(u_2)>f(u_3)$. Iterating this procedure, the finiteness of the graph ensures that a path $\Gamma$ such that $u_n\in\boundary$ or $f(u_n)=-f(u_1)$ exists.
  The following equalities
  \begin{equation*}
    \frac{f(u_i)-f(u_{i+1})}{\|\grad f\|_{\infty}}=\frac{1}{\edgelength_{u_i,u_{i+1}}}\,\:\forall i=1,\dots,n-1
  \end{equation*}
  provide an explicit expression for the length of the path $\Gamma$:
  \begin{equation}\label{eq1_optimal_paths_limiting_eigenfunctions}
    \length(\Gamma)=
    \sum_i \frac{1}{\edgelength_{u_i,u_{i+1}}}=
    \frac{f(u_1)-f(u_n)}{\|\grad f\|_{\infty}}\,.
  \end{equation}
  This last equation, together with the equality $\Lambda=\|\grad f\|_{\infty}/|f(u_1)|$, yields $\Lambda=1/\length(\Gamma)$ if $u_n\in\boundary$ and $\Lambda=2/\length(\Gamma)$ if $f(u_n)=-f(u_1)$\,.

  To conclude we are left to prove that i) if $u_n\in\boundary$ then  $\length(\Gamma)=d_{\boundary}(u_1)$ and ii) $\length(\Gamma)\leq 2d(u_1,v)$ for all $v$ such that $f(v)=-f(u_1)$, together with the analogous expressions for the case $f(u_n)=-f(u_1)$.
  To this aim, assume by contradiction that there exists a path $\Gamma'=\{(v_i,v_{i+1})\}_{i=1}^{m-1}$ with $v_1=u_1$, $v_m\in\boundary$, and 
  $\sum_{i=1}^{m-1}\edgelength_{v_i v_{i+1}}^{-1}<\sum_{i=1}^{n-1}\edgelength_{u_i u_{i+1}}^{-1}$ and observe that
  \begin{equation*}
    \edgelength_{v_iv_{i+1}}|f(v_i)-f(v_{i+1})|=
    |\grad f(v_i,v_{i+1})|\leq\|\grad f\|_{\infty},
    \quad \forall i=1,\dots,m-1\,,
  \end{equation*}
  Then, using the triangular inequality and then equality~\eqref{eq1_optimal_paths_limiting_eigenfunctions},
  and the assumptions $u_n, v_m\in\boundary$ (i.e. $f(u_n)=f(u_m)=0$), we obtain the following contradiction:
  \begin{equation*}
  \begin{aligned}
    |f(u_1)|\leq\sum_{i=1}^{m-1}|f(v_i)-f(v_{i+1})|
    &\leq\|\grad f\|_{\infty}
    \sum_{i=1}^{m-1}\edgelength_{v_i v_{i+1}}^{-1}\\
    &<\|\grad f\|_{\infty}\sum_{i=1}^{n-1}\edgelength_{u_i u_{i+1}}^{-1}
    =|f(u_1)|\,.
  \end{aligned}
  \end{equation*}
  The inequality $\length(\Gamma)\leq 2\,d(u_1,v)$ for all $v$ such that $f(v)=-f(u_1)$, can be proved with a similar argument, i.e., assuming the existence of a path $\Gamma'$, from $u_1$ to $v$, with $f(v)=-f(u_1)$, such that $\length(\Gamma')<2\,\length(\Gamma)$. The final case $f(u_n)=-f(u_1)$ can be treated analogously.
\end{proof}

We now turn our attention to the study of the accumulation points of sequences of $p$-Laplacian variational eigenvalues. We start by introducing some necessary preliminary results that are of independent interest. First of all we observe that the minimum of the $\infty$-Rayleigh quotient
\begin{equation}\label{minimizing_inf_rayleigh_quotient_problem}
  \rayl_{\infty}(f)=\frac{\| \grad f\|_{\infty}}{\| f \|_{\infty}}
\end{equation}
can be easily computed, as in the continuous case~\cite{Lind3,BungertInfinityL2}, considering the distance from the boundary, $d_{\boundary}$ given in~\eqref{boundary_distance}.
Before proving this result, let us consider the following trivial inequality. Assume the distance between the nodes $u$ and $v$ is realized along the path $\Gamma=\{(v_i,v_{i+1})\}_{i=1}^{n-1}$, with $v_1=u$ and $v_n=v$, i.e.,  $d(u,v)=\sum_{i=1}^m 1/\edgelength_{v_iv_{i+1}}\,.$ 
Then, the trianguar inequality yields the following Lipschitz condition:
\begin{equation}\label{Lipschitzianity}
  \begin{aligned}
    |f(u)-f(v)|
    &
      \leq |\sum_{i=1}^m |f(v_i)-f(v_{i+1})|
      \leq \|\grad f\|_{\infty}\sum_{i=1}^m\frac{1}{\edgelength_{v_iv_{i+1}}}
      =\|\grad f\|_{\infty}d(u,v)\,.
  \end{aligned}
\end{equation}
The next result shows that the boundary distance function $d_{\boundary}$ defined in~\eqref{boundary_distance} is a minimizer of the Rayleigh quotient.
\begin{proposition}\label{Prop_boundary_distance_is_a_minimizer}
  The boundary distance function $d_B\in\nodeset$ realizes the minimunm of the $\infty$-Rayleigh quotient:
  \begin{equation*}
    d_{B}\in
    \argmin_{f\in\mathcal{H}_0(\nodeset)}
    \frac{\|\grad f\|_{\infty}}{\| f \|_{\infty}}\,,
  \end{equation*}
\end{proposition}
\begin{proof}
  First we prove that $\| \grad d_{\boundary} \|_{\infty}\leq 1$. Indeed, given a node $u\in\nodeset$ and an edge $(u,v)\in\edgeset$ connected to it, the gradient of the distance function $d_{\boundary}$ on this edge can be estimated as:
  \begin{equation*}
    \grad d_{\boundary}(u,v)
    = \edgelength_{uv}\big(d_{\boundary}(v)-d_{\boundary}(u)\big)
    \leq \edgelength_{uv}(\edgelength_{uv})^{-1}=1\,.
  \end{equation*}
  Thus, to conlude it is sufficient to show that the boundary distance function satisfies:
  \begin{equation*}
    d_{\boundary}\in\argmax_{f\in\mathcal{F}:\| \grad f\|_{\infty}\leq 1}\|f\|_{\infty}\,.
  \end{equation*}
  To this end, for any function $f$ with $f\in\mathcal{F}$ and $\|\grad f\|_{\infty}\leq 1$ and for any node $u\in\nodeset$, let $v\in\boundary$ be the boundary node such that $d(u,v)=d_{\boundary}(u)$. Then, using~\eqref{Lipschitzianity} we can write:
  \begin{equation*}
    |f(u)|=|f(u)-f(v)|\leq d(u,v)=d_{\boundary}(u)\,.
  \end{equation*}
  The last equation clearly implies:
  \begin{equation*}
    d_{\boundary}\in
    \argmax_{f\in\mathcal{H}_0(\nodeset):\|\grad f\:\|_{\infty}\leq 1}\|f\|_{\infty}.
  \end{equation*}
\end{proof}

Next we show that the minimum value of the Rayleigh quotient is equal to the radius of the largest ball that can be inscribed in $\Gc$.
\begin{corollary}\label{First_infinity_eigenvalue}
  \begin{equation*}
    \min_{f\in\mathcal{H}_0(\nodeset)} \frac{\| \grad f\|_{\infty}}{\| f\|_{\infty}}
    =\frac{1}{R_1}\,,
  \end{equation*}
  where $R_1$ is defined as 
    $R_1=\sup_{u\in \internalnodes}\lbrace d_{\boundary}(v) \rbrace$.
\end{corollary}
\begin{proof}
  Because of Proposition~\ref{Prop_boundary_distance_is_a_minimizer}, it is enough to compute $\|\grad d_{\boundary}\|_{\infty}$ and $\| d_{\boundary}\|_{\infty}$.
  Trivially $\|d_{\boundary}\|_{\infty}=R_1$.
  From the proof of Proposition \ref{Prop_boundary_distance_is_a_minimizer}, we already know that $\big|\grad d_{\boundary}(u,v)\big|\leq 1 \;\: \forall (u,v)\in E$. Moreover, it is easy to prove (see e.g. \cite{BungertInfinityL2}) that:
  \begin{equation*}
    |\grad d_{\boundary}(u,v)|= 1
    \quad \Longleftrightarrow \quad
    v\in\Gamma(u,\boundary) \,\: or\, u\in\Gamma(v,\boundary)\,,
  \end{equation*}
  where $\Gamma(u,\boundary)$ is the shortest path from the node $u$ to the boundary. Thus, necessarily $\|\grad d_{\boundary}\|_{\infty}=1$.
\end{proof}

%
The next theorem shows that, given $(f,\Lambda)$ solution to the $\infty$-limit eigenvalue eq.~\eqref{limitin_inf_eigenvalue_eq}, the eigenvalue $\Lambda$ can be regarded as the maximal radius that allows one to inscribe as many balls in the graph as the number of nodal domains induced by $f$.
Recall that the nodal domains induced by $f$ are the maximal connected subgraphs where $f$ is strictly positive or negative.
The proof of the theorem relies on results related to the $p$-Laplacian taken from~\cite{Tudisco1, DEIDDA2023_nod_dom}, more specifically when $p=1$ as specialized in~\cite{ZhangNodalDO}
%
%
\begin{theorem}\label{Thm_Nodal_domains_balls_inscribed}
  Let $(f,\Lambda)$ be an eignepair that satisfies eq.~\eqref{limitin_inf_eigenvalue_eq}, and assume that $f$ induces $k$ nodal domains. Then there exist $v_1,\dots ,v_k\in\nodeset$ and $r_k>0$ such that 
  \begin{equation*}
    d(v_i,v_j)\geq 2r_k \quad 
    \quad d_{\boundary}(v_i)\geq r_k
    \quad \text{ for all } i,j=1,\dots,k,  i\neq j\,.
  \end{equation*}
  Moreover  
  \begin{equation*}
    \Lambda=\frac{1}{r_k}\,.
  \end{equation*}
\end{theorem}
\begin{proof}
  Consider first the case of $f$ strictly positive. Then, given $v_1\in\nodemaxset(f)$, from Proposition \ref{optimal_paths_limiting_eigenfunctions} there exists a path $\Gamma=\{(u_i,u_{i+1})\}_{i=1}^{n-1}$ starting from $u_1=v_1$ such that $u_{n}\in\boundary$ and $\length(\Gamma)=d_{\boundary}(v_1)=\Lambda^{-1}$.
  Thus, the node $v_1$ and the radius $r_1=\Lambda^{-1}$ satisfy the thesis.
 
  Consider now the case of a function $f$ satisfying~\eqref{limitin_inf_eigenvalue_eq} and assume $\{A_i\}$, $\{B_j\}$ to be its nodal domains in such a way that:
  \begin{equation*}
    f(u)>0\quad  \forall u\in \cup_{i}A_i\,,\qquad f(u)<0\quad \forall u\in\cup_jB_j\,.    
  \end{equation*}
  Starting from the original graph $\Gc$ and the function $f$, we aim at defining a new disconnected graph $\Gc'$ such that:
  \begin{itemize}
  \item $\Gc'$ has as many connected components, $\Gc_h$, as the nodal domains of $f$,
  \item for any $h$, $(f|_{\Gc_h},\Lambda)$ satisfy the $\infty$-limit eigenvalue equation \eqref{limitin_inf_eigenvalue_eq} on $\Gc_h$,
  \item $f|_{\Gc_h}$ is strictly positive or strictly negative.
  \end{itemize}  
  To this end, we consider all the nodes $u$ such that $f(u)=0$ as boundary nodes of the new graph, $u\in\boundary(\Gc')$. 
  Furthermore, for any edge $(u,v)$ such that $f(u)f(v)<0$, we add a boundary node $w\in\boundary(\Gc')$ and replace the edge $(u,v)$ by two edges $(u,w)$ and $(w,v)$ with weights $\edgelength_{uw}=\edgelength_{uv}\big(1-f(v)/f(u)\big)$ and $\edgelength_{wv}=\edgelength_{uv}\big(1-f(u)/f(v)\big)$. Observe that, since $f(w)=0$ ($w\in\boundary(\Gc')$), the gradient on $(u,v)$ is the same as the gradients on $(u,w)$ and $(w,v)$, i.e.:
  \begin{equation}\label{gradient_unchanged}
    \begin{aligned}
      \grad f(u,v)
      =\edgelength_{uv}\big(f(v)-f(u)\big)
      =\edgelength_{uw}\big(0-f(u)\big)
      =\grad f(u,w),\\
      \grad f(v,u)
      =\edgelength_{uv}\big(f(u)-f(v)\big)
      =\edgelength_{vw}\big(0-f(v)\big)
      =\grad f(v,w),
    \end{aligned}
  \end{equation}
  and the weights are conjugate:
  \begin{equation}\label{distances_unchanged}
    \frac{1}{\edgelength_{uw}}+\frac{1}{\edgelength_{vw}}
    =\frac{1}{\edgelength_{uv}}\,.
  \end{equation}
  Note that the internal nodes of the new graph $\Gc'$ are the nodes of $\Gc$ where $f$ is non zero. 
  In addition, $\Gc'=(\cup_{h=1}^n \Gc_h)\cup\boundary(\Gc')$ with $\big(\Gc_{h_1}\cap\Gc_{h_2}=\emptyset$ for any pair of distinct indices $h_1$ and $h_2$ varying between 1 and $n$.  Moreover, every $\Gc_h$ matches one of the nodal domains induced by $f$, meaning that its internal nodes are all and only the nodes that belong to some nodal domain induced by $f$ on $\Gc$. 
 
  We now observe that for any internal node $u$ of $\Gc'$, from \eqref{gradient_unchanged}, the local gradient, $\Opc{\grad f}(u)$ is unchanged, implying that $(f|_{\Gc_h},\Lambda)$ satisfies the $\infty$-limit equation \eqref{limitin_inf_eigenvalue_eq} with respect to the graph $\Gc_h$.
  Thus, since $f|_{\Gc_h}$ is strictly positive or negative, the first part of the proof guarantees that, for any $h$, there exists $v_h\in\Gc_h$ such that 
  \begin{equation*}
    \Lambda=\frac{1}{d_{\boundary(\Gc')}(v_h)}\,.
  \end{equation*}
  We conclude the proof of the Theorem by observing that:
  \begin{equation*}
    \frac{2}{\Lambda}=d_{\boundary(\Gc')}(v_{h_1})+d_{\boundary(\Gc')}(v_{h_2})\leq d(v_{h_1},v_{h_2})\,.
  \end{equation*}
  To prove this last assertion, we first recall that $\Gc_{h_1}\cap\Gc_{h_2}=\emptyset$ and that, by construction eq.~\eqref{distances_unchanged} holds, i.e., we have not changed distances among internal nodes.
  Now, we let $\Gamma$ be the shortest path that joins $v_{h_1}$ and $v_{h_2}$ in $\Gc$, i.e. $d(v_{h_1},v_{h_2})=\length(\Gamma)$. 
  Then, in $\Gc'$, the path $\Gamma$ has necessarily been replaced by some path $\Gamma'$ that crosses $\boundary(\Gc')$ and such that $\length(\Gamma')=\length(\Gamma)$. 
\end{proof}

Finally, we conclude this section by studying  the relationships between the limits of the $p$-Laplacian variational eigenvalues, $\{\Lambda_k\}_k$, and the packing radii of $\Gc$ \cite{grove1995new}, $\{R_k\}_k$, i.e. the maximal radii that allow one to inscribe a prescribed number of disjoint balls in the graph.
We start by defining the $k$-th packing radius of the graph $\Gc$ as follows.
\begin{definition}[$k$-th packing radius $R_k$]\label{k-th_radius}
  \begin{multline*}
    R_k:=\max\left\{ r>0:\exists\, u_1,\dots,u_k\in\internalnodes,
      d(u_i,u_j)\geq 2r, \, d_{\boundary}(u_i)\geq r \:\forall\, i,j=1,\dots,k \right\}\!.
  \end{multline*}
\end{definition}
 Theorem~\ref{Thm_Nodal_domains_balls_inscribed} trivially implies the following corollary, which we state without proof.
\begin{corollary}\label{Corollary_lower_bound_radius_number_of_nodal_domains}
  Let $(f,\Lambda)$ be an eigenpair satisfying eq.~\eqref{limitin_inf_eigenvalue_eq} and assume that $f$ induces $k$ nodal domains.  Then:
  \begin{equation*}  
   \Lambda\geq\frac{1}{R_k}\,.
  \end{equation*}
\end{corollary}
From this corollary it follows immediately that $\Lambda=1/r_k$ with $r_k\leq R_k$.
We proceed now by looking at upper bounds for $\Lambda_k$. Recalling that the symbol  $\Lambda_k$ is used to denote \new{the limit of the sequence} $\lambda_{k}^{\frac{1}{p}}(\plap)$, we have the following Proposition relating $p$-Laplace eigenvalues and packing radii.
\begin{proposition}\label{Variational_eienvalues_radius_estimates} 
  Let $\lambda_k(\plap)$ be the $p$-Laplacian $k$-th variational eigenvalue on a graph $\Gc$ and let $R_k$ be the $k$-th packing radius of $\Gc$. 
  Then, 
  \begin{equation*}
   \new{\Lambda_k=}\lim_{p\rightarrow\infty}\lambda_k(\plap)^{\frac{1}{p}}\leq \frac{1}{R_k}\,.
  \end{equation*}
\end{proposition}
\begin{proof}
\new{We recall that the existence of the limit $\lim_{p\to \infty}\lambda_k(\plap)^\frac{1}{p}$ is guaranteed from eq.~\eqref{eq_monotonocity_variational_eigenvalue} and \cite{deidda2025nonlinear}.} Then, let $u_1,\dots,u_k$ be $k$ internal nodes of $\Gc$ as in Definition \ref{k-th_radius}, i.e.:
  \begin{equation*}
    d(u_i,u_j)\geq 2R_k\,, \qquad d_{\boundary}(u_i)\geq R_k\quad \forall i,j=1,\dots,k\,.
  \end{equation*}
  We define $k$ linearly independent functions as:
  \begin{equation*}
    f_i(u)=\max\lbrace R_k-d(u,u_i),0\rbrace\,.
  \end{equation*}
  Obviously, the set $A_k$ spanned by these functions, $A_k=\Span\lbrace f_i \rbrace_{i=1}^k$, has dimension $k$ and also its Krasnoselskii genus $\gamma(A_k)$ is equal to $k$.
  Then, by definition, we have:
  \begin{multline*}
    \lambda_k(\plap)
    \leq\max_{f\in A_k}\mathcal{R}_{\plap}(f)
    =\max_{f\in A_k}\frac{\sum_{(u,v)\in E}\edgelength_{uv}^p| f(u)-f(v)|^p}%
    {2\sum_{u\in\nodeset}|f(u)|^p}\\
    =\frac{
      \sum_{(u,v)\in E}\edgelength_{uv}^p
      \big|\sum_{i=1}^k\alpha_i f_i(u)-\sum_{j=1}^k \alpha_j f_j(v)\big|^p}%
    {2\sum_{m=1}^k \left(\sum_{d(u,u_m)<R_k}|\alpha_mf_m(u)|^p\right)}\,,
  \end{multline*}
  where we recall the factor $1/2$ is in the definition of the $p$-norm of $\grad f$.
  If $u$ and $v$ are such that both $d(u,u_i)$ and $d(v,u_j)$ are smaller than $R_k$, then:
  \begin{align*}
    \edgelength_{uv}| f(u)-f(v)|
    &= \edgelength_{uv}|\alpha_i||f_i(u)-f_i(v)\big|
      = \edgelength_{uv}|\alpha_i||d(u,u_i)-d(v,u_i)\big|\\
    &\leq \edgelength_{uv}|\alpha_i|d(u,v)
      \leq \edgelength_{uv}|\alpha_i|\frac{1}{\edgelength_{uv}}
      \leq |\alpha_i|\,.
  \end{align*}
  Instead, if $d(u,u_i)<R_k$ and $d(v,u_j)<R_k$ with $i\neq j$, then:
  \begin{align*}
    \edgelength_{uv}| f(u)-f(v)|
    &=\edgelength_{uv}|\alpha_i\,f_i(u)-\alpha_j\,f_j(v)|\\
    &\leq \edgelength_{uv}\,\max\{|\alpha_i|,\,|\alpha_j|\}\,
      \big(R_k-d(u,u_i)+R_k-d(v,u_j)\big)\\
    &\leq \edgelength_{uv}\,\max_l\{|\alpha_l|\}
      \big(2R_k-d(u_i,u_j)+d(u,v)\big)\\
    &\leq \edgelength_{uv}\,\max_l\{|\alpha_l|\}
      \big(2R_k-2R_k+d(u,v)\big)\\
    &\leq \edgelength_{uv}\,\max_l\{|\alpha_l|\}
      d(u,v)\\
    &\leq \max_l\{|\alpha_l|\}\,.
  \end{align*}
  Lastly, if $d(u,u_i)<R_k$ and $d(v,u_j)\geq R_k$ for all $j$, we have:
  \begin{align*}
    \edgelength_{uv}|f(u)-f(v)|=
    &\edgelength_{uv}|\alpha_i|\big(R_k-d(u,u_i)\big)\\
    &=\edgelength_{uv}|\alpha_i|\big(R_k-d(u,u_i)+ d(u,v)-d(u,v)\big)\\
    &\leq \edgelength_{uv}|\alpha_i|\big(d(u,v)+R_k-d(v,u_i)\big) \\
    &\leq \edgelength_{uv}|\alpha_i|d(u,v)\\
    &\leq |\alpha_i|\,.
  \end{align*}
  Using the above inequalities in the expression for $\lambda_k$, we can write:
  \begin{equation*}
    \frac{\sum_{(u,v)\in E}\edgelength_{uv}^p
      \big|\sum\alpha_if_i(u)-\sum \alpha_if_i(v)\big|^p}%
    {2\sum_{m=1}^k \big(\sum_{d(u,u_m)<R_k}|\alpha_mf_m(u)|^p\big)}
    \leq
    \frac{\sum_{(u,v)\in E}\max_i|\alpha_i|^p}%
    {2\sum_{m=1}^k\big(\sum_{d(u,u_m)<R_k}|\alpha_mf_m(u)|^p\big)}\,.
  \end{equation*}
  Passing to the limit we arrive at:
  \begin{align*}
    \limsup_{p\rightarrow\infty} \lambda_k^{\frac{1}{p}}(\plap)
    &\leq
    \limsup_{p\rightarrow\infty}
    \bigg(
    \frac{\sum_{(u,v)\in E}\max_i|\alpha_i|^p}%
    {2\sum_{m=1}^k \big(\sum_{d(u,u_m)<R_k}|\alpha_mf_m(u)|^p\big)}
    \bigg)^{\frac{1}{p}}\\
    &=\frac{\max_i|\alpha_i|}{\max_{m,u} |\alpha_m f_m(u)|}\,.
  \end{align*}
  The thesis is obtained by observing that $f_i(u_i)= R_k$ and thus $\max_{m,u}|\alpha_{m} f(u)|= R_k\max_m|\alpha_m|$.
\end{proof}
 Next, we show the main theorem of this section, which provides a geometrical characterization of the first two  variational eigenvalues of the $p$-Laplacian operator showing that they are exactly equal to the reciprocal of the first two packing radii.
\begin{theorem}\label{thm:packing-radii}
  Assume $\Gc$ to be a connected graph and $\Lambda_k:=\lim_{p_j\rightarrow \infty}\lambda_k(\Delta_{p})^{\frac{1}{p}}$, $k=1,2$. Then 
  \begin{equation*}
    \Lambda_k=\frac{1}{R_k}\quad k=1,2\,.
  \end{equation*}
\end{theorem}
\begin{proof}
  The proof for the first eigenvalue ($k=1$) follows easily from Proposition \ref{Variational_eienvalues_radius_estimates}, Corollary \ref{First_infinity_eigenvalue}, and the observation that $\lambda_1(\plap)=\min_{f\in \mathcal{H}_0(\nodeset)}\rayl_p(f)$.

  For the second eigenvalue ($i=2$) we proceed as follows. Proposition \ref{Variational_eienvalues_radius_estimates} ensures that $\Lambda_2\leq \frac{1}{R_2}$.
  Then we can consider a subsequence indexed by $p_h$ of convergent eigenpairs, i.e.:
  \begin{equation*}
    \big(f_{2}(\Delta_{p_h}),\lambda_2(\Delta_{p_h})\big)\rightarrow (f,\Lambda_2)\,.    
  \end{equation*}
  From \cite{Tudisco1,DEIDDA2023_nod_dom}, we know that, for any $p_h$, $f_2(\Delta_{p_h})$ has at least two nodal domains, which we denote by $A_{p_h}$ for $u\in\internalnodes$ such that $f_2(\Delta_{p_h})(u)>0$, and $B_{p_h}$ for $u\in\internalnodes$ such that $f_2(\Delta_{p_h})(u)<0$.
  Then the sets 
  $A=\cap_n\cup_{p_h>n}A_{p_h}$ %
  and 
  $B=\cap_n\cup_{p_h>n}B_{p_h}$
  are both non empty and such that $f(u)\geq 0$ for any $u\in A$ and $f(u)\leq 0$ for any $u\in B$.
  If by contradiction $\{u\:|\;f(u)>0\}=\emptyset$, since $f\neq0$, there has to exist a node $u$ with $f(u)=0$ that is connected to a node $v\sim u$ such that $f(v)<0$ that means: 
  \begin{equation*}
    \inflap f(u)=\|\Opc{\grad f}^-(u)\|_{\infty}-\|\Opc{\grad f}^{+}(u)\|_{\infty}=\|\Opc{\grad f}^-(u)\|_{\infty}>0\,.
  \end{equation*}
  But this is an absurd because $f$ has to satisfy eq.~\eqref{limitin_inf_eigenvalue_eq}, implying that $f$ must induce at least two nodal domains.
  Then, thanks to Corollary \ref{Corollary_lower_bound_radius_number_of_nodal_domains}, we get:
  \begin{equation*}
    \Lambda_2=\frac{1}{r_2}\geq \frac{1}{R_2}\,,
  \end{equation*}
  which concludes the proof.  
\end{proof}

Finally, the following noteworthy Corollary provides lower and upper bounds to the $k$-th variational $\infty$-eigenvalues in terms of packing radii and nodal domains of the corresponding $k$-th eigenfunctions.
  The proof is a trivial consequence of Proposition \ref{Variational_eienvalues_radius_estimates}, Corollary \ref{Corollary_lower_bound_radius_number_of_nodal_domains} and Theorem \ref{limitin_infty_eigenvalue_theorem}.
%

\begin{corollary}\label{cor:lower-upper}
  Let $(f_{k,p_j},\lambda_{k}(\Delta_{p_j}))$ be a sequence of the $k$-th variational $p$-Laplacian eigenpairs. If $
    \lim_{j\rightarrow \infty}(f_{k,p_j},\lambda_{k}^{1/{p_j}}(\Delta_{p_j}))
    =(f_k,\Lambda_k)\,,$
  then 
  \begin{equation*}
    \frac{1}{R_{\mathcal{N}(f_k)}}\leq\Lambda_k\leq\frac{1}{R_K}\,,
  \end{equation*}
  where $\mathcal{N}(f_k)$ is the number or nodal domains induced by $f_k$.
\end{corollary}

\begin{remark}
  It is worthwhile to spend a short remark about the non-boundary case, which is analogous to the homogeneous Neumann case in the continuum setting \cite{Esposito}. In this case, we can think of the Neumann case as a Dirichlet problem defined on a graph $\Gc$ having a boundary set $\boundary$ formed by nodes $v$ at an infinite distance from any internal node, so that the edge weights $\edgelength_{uv}$ are zero for all edges such that $u\in\internalnodes$ and $v\in\boundary$. Then it is clear that in this case $\Lambda_1=0$ and
  \begin{equation*}
    \Lambda_2=\frac{1}{R_2}:=1/\sup\lbrace r \mathrm{ \;\;s.t.\; \;} \exists v_1,v_2 \;\:\mathrm{s.t.}\;\: d(v_1,v_2)\geq 2r \rbrace\,,
  \end{equation*}
  i.e. half the diameter of the graph. Moreover,
  for the higher eigenvalues, similarly to Proposition \ref{Variational_eienvalues_radius_estimates}, we can state
  that:
  \begin{equation*}
    \Lambda_k^{-1}\geq \sup\lbrace r>0:\; \exists\; v_1,\ldots,v_k\in\internalnodes :\; d(v_i,v_j)\geq 2r \;\: \forall i,j=1,\dots,k \rbrace\,.
  \end{equation*}
\end{remark}

We conclude this section by producing some examples. We first show that the solution of \eqref{limitin_inf_eigenvalue_eq} is not always well defined and that not any solution of \eqref{limitin_inf_eigenvalue_eq} is also the limit of $p$-Laplacian eigenfunctions.
\begin{NEW}
In the second example we investigate the properties of the boundary distance function $d_{\boundary}$ as a minimizer of the infinity Rayleigh quotient $\rayl_{\infty}$. Indeed, even if $d_{\boundary}$  always realizes the minimum of $\rayl_{\infty}$, generically it is not the unique minimizer and it may not satisfy the limit eigenvalue equation~\eqref{limitin_inf_eigenvalue_eq}, (see also \cite{bungert2021eigenvalue, deidda2025nonlinear, Lind2, Yu}).
\end{NEW}

\begin{example}\label{ex:1} 
  Consider the graph of figure \ref{fig_ex_1}.
\begin{figure}[h]
  \begin{center}
    \begin{tikzpicture}[x={(\unitlength,0)},y={(0,\unitlength)}]
      \coordinate (Origin) at (0,0);
      \def\xscale{40}
      \def\yscale{40}
      \coordinate (Xone) at (\xscale,0);
      \coordinate (Yone) at (0,\yscale);
      \coordinate (N1) at ($(Origin)$);
      \coordinate (N2) at ($(Origin)+2*(Xone)$);
      \coordinate (N3) at ($(Origin)+3*(Xone)$);
      \coordinate (N4) at ($(Origin)+4*(Xone)$);
      \coordinate (N5) at ($(Origin)+5*(Xone)$);
      \coordinate (N6) at ($(Origin)+6*(Xone)$);
      \coordinate (N7) at ($(Origin)+8*(Xone)$);
      \coordinate (N8) at ($(Origin)+4*(Xone)+1*(Yone)$);
      \draw[black]
      (N1) node[shape=circle,draw,inner sep=0.5pt, fill=white]
      {{$\boundary$}};
      \draw[black]
      (N2) node[shape=circle,draw,inner sep=0.5pt, fill=white]
      {{$u_1$}};
      \draw[black] 
      (N3) node[shape=circle,draw,inner sep=0.5pt, fill=white]
      {{$u_2$}} ;
      \draw[black]
      (N4) node[shape=circle,draw,inner sep=0.5pt, fill=white]
      {{$u_3$}};
      \draw[black]
      (N5) node[shape=circle,draw,inner sep=0.5pt, fill=white]
      {{$u_4$}};
      \draw[black]
      (N6) node[shape=circle,draw,inner sep=0.5pt, fill=white]
      {{$u_5$}};  
      \draw[black]
      (N7) node[shape=circle,draw,inner sep=0.5pt, fill=white]
      {{$\boundary$}};
      \draw[black]
      (N8) node[shape=circle,draw,inner sep=0.5pt, fill=white]
      {{$\boundary$}}; 
      \begin{scope}[on background layer]
        \draw[black]
        (N1)--(N2) node[pos=0.5,sloped,above]
        {\scriptsize{$\edgelength_{1\boundary}\!=\!1$}};;
        \draw[black]
        (N2)--(N3) node[pos=0.5,sloped,above]
        {\scriptsize {$\edgelength_{12}\!=\!2$}};
        \draw[black]
        (N3)--(N4)node[pos=0.5,sloped,above]
        {\scriptsize {$\edgelength_{23}\!=\!2$}};
        \draw[black]
        (N4)--(N5) node[pos=0.5,sloped,above]
        {\scriptsize {$\edgelength_{34}\!=\!2$}};;
        \draw[black]
        (N5)--(N6) node[pos=0.5,sloped,above]
        {\scriptsize {$\edgelength_{45}\!=\!2$}};
        \draw[black]
        (N6)--(N7)node[pos=0.5,sloped,above]
        {\scriptsize {$\edgelength_{5\boundary}\!=\!1$}};
        \draw[black]
        (N4)--(N8)node[pos=0.5,left]
        {\scriptsize {$\edgelength_{3\boundary}\!=\!2$}};
      \end{scope}
    \end{tikzpicture}
  \end{center}
  \caption{Graph of Example~\ref{ex:1}.}\label{fig_ex_1}
\end{figure}

  Easy computations show that both functions $f$ and $g$ given by:
  \begin{equation*}
    f=\left(1,\frac{2}{3},\frac{1}{3},\frac{2}{3},1\right)
  %
  \qquad\mbox{and}\qquad 
  %
    g=\left(1,\frac{2}{3},\frac{1}{3},\frac{2}{3},\frac{4}{9}\right)
  \end{equation*}
  satisfy the $\infty$-limit eigenvalue equation~\eqref{limitin_inf_eigenvalue_eq} with $\Lambda=1$.
  Nevertheless, $g$ cannot be the limit of first eigenfunctions of the $p$-Laplacian, which, because of their uniqueness and the symmetry of the above graph, must be symmetric for any $p$.
\end{example}

\begin{NEW}

\begin{example}\label{ex:2}

  Consider the path graph of Figure \ref{fig_ex_2}.
  \begin{figure}[h]
  \begin{center}
    \begin{tikzpicture}[x={(\unitlength,0)},y={(0,\unitlength)}]
      \coordinate (Origin) at (0,0);
      \def\xscale{40}
      \def\yscale{40}
      \coordinate (Xone) at (\xscale,0);
      \coordinate (Yone) at (0,\yscale);
      \coordinate (N1) at ($(Origin)$);
      \coordinate (N2) at ($(Origin)+1.5*(Xone)$);
      \coordinate (N3) at ($(Origin)+3.75*(Xone)$);
      \coordinate (N4) at ($(Origin)+6*(Xone)$);
      \draw[black]
      (N1) node[shape=circle,draw,inner sep=0.5pt, fill=white]
      {{$\boundary$}};
      \draw[black]
      (N2) node[shape=circle,draw,inner sep=0.5pt, fill=white]
      {{$u_1$}};
      \draw[black] 
      (N3) node[shape=circle,draw,inner sep=0.5pt, fill=white]
      {{$u_2$}} ;
      \draw[black]
      (N4) node[shape=circle,draw,inner sep=0.5pt, fill=white]
      {{$\boundary$}};
      \begin{scope}[on background layer]
        \draw[black]
        (N1)--(N2) node[pos=0.5,sloped,above]
        {\small {$\edgelength_{1\boundary}\!=\!3$}};;
        \draw[black]
        (N2)--(N3) node[pos=0.5,sloped,above]
        {\small {$\edgelength_{12}\!=\!2$}};
        \draw[black]
        (N3)--(N4)node[pos=0.5,sloped,above]
        {\small {$\edgelength_{2B}\!=\!2$}};
      \end{scope}
    \end{tikzpicture}
  \end{center}
  \caption{Graph of Example~\ref{ex:2}.}\label{fig_ex_2}
\end{figure}

The boundary distance of $u_1$ and $u_2$ is given by:
  \begin{equation*}
    d_{\boundary}(u_1)=\frac{1}{3},\quad
    d_{\boundary}(u_2)=\frac{1}{2}\,.    
  \end{equation*}

  Some easy computations show that
  \begin{equation*}
  \begin{aligned}
    \inflap d_{\boundary}(u_1)=1-2\Big(\frac{1}{2}-\frac{1}{3}\Big)=1-\frac{1}{3}\gneq 0\\
    \|\Opc{\grad d_{\boundary}}(u_1)\|_{\infty}-\Lambda_1 d_{\boundary}(u_1)=1- \frac{2}{3}\gneq 0\,.
  \end{aligned}
  \end{equation*}
   Thus $d_{\boundary}$ does not satisfy the $\infty$-limit eigenvalue equation \eqref{limitin_inf_eigenvalue_eq}, and in particular it can not be the limit of the first eigenfunctions of the $p$-Laplacian. 
   Additionally, it is possible to check that, besides $d_{\boundary}$, any function $f$ with 
   \begin{equation}
   f(u_1)\in\Big[0,\frac{1}{3}\Big],\qquad  f(u_2)=\frac{1}{2},
   \end{equation}
   is such that $f\in \argmin_{g\in \mathcal{H}_0(\nodeset)} \rayl_{\infty}(g)$. Moreover, it is easy to see that the function $f^*(u_1)=1/5$, $f^*(u_2)=1/2$  is the unique solution, up to nonzero constant multiplicative factors, of the $\infty$-limit eigenvalue  equation \eqref{limitin_inf_eigenvalue_eq} with $\Lambda=1/2$.
\end{example}
Next, we characterize all the graphs such that the boundary distance function is the unique minimizer of the $\infty$-Rayleigh quotient. In particular, this provides sufficient conditions for the boundary distance function to be the unique solution of the $\infty$-limit eigenvalue equation \eqref{limitin_inf_eigenvalue_eq}. Sufficient conditions for the boundary distance function to be the unique solution of the $\infty$-limit eigenvalue equation have been investigated also in the continuous setting in \cite{Yu}.

\begin{theorem}\label{thm_distance_unique_minimizer}
    Let $\Gc$ be a connected graph with boundary $\boundary$. Then the boundary distance function is the unique minimizer of $\rayl_\infty$, up to nonzero multiplicative constant factors, if and only if given $u_0\in \argmax_{u\in\internalnodes} d_{\boundary}(u)$ for any $v\in \internalnodes$ there exists some path $\gamma=\{u_0,u_1,\dots,u_n|\; u_n\in \boundary\}$ from $u_0$ to the boundary such that $u_i=v$ for some $1\leq i\leq n$ and $\length
(\gamma)=d_\boundary(u_0)$. 
\end{theorem}
\begin{proof}
We start proving that the condition on the paths is a sufficient condition for the uniqueness of the minimizer. Let $f\in \argmax_{g}\rayl_{\infty}(g)$ and, without loss of generality, assume 
\begin{equation*}
    \|\grad f\|_{\infty}=1  \quad \text{and} \quad f(u_0)=d_{\boundary}(u_0).    
\end{equation*}
Consider $v\in \internalnodes$. Since the graph is connected, there exists some path $\gamma'=\{u_0,u_1,\dots, u_n\}$ such that $u_n\in \boundary$, $v=u_k$ for some $k\in \{1,\dots, n-1\}$ and $\length(\gamma')=d_{\boundary}(u_0)$. In particular 
\begin{equation}
d_{\boundary}(v)=\length(\{u_k,\dots, u_n\})\quad  \text{and} \quad d(v,u_0)=\length(\{u_0,\dots, u_k\}).
\end{equation}
Thus, as a consequence of \eqref{eq1_optimal_paths_limiting_eigenfunctions}, necessarily $f(v)=d_{\boundary}(v)$. 

On the other hand, let 
\begin{equation}
    V^*=\left\{u\in \internalnodes\;\Bigg|\!\!\!\!\!\!\begin{array}{ll}&\exists \gamma=\{u_0,u_1,\dots, u_n\in \boundary\} \text{ with } d_{\boundary}(v)=\length(\gamma)\\
         &\text{ and } u=u_i \text{ for some } 1\leq i\leq n-1
    \end{array}\right\}\cup \boundary
\end{equation}    
be the set of all nodes of the graph supported on some shortest path from $u_0$ to the boundary. Then assume $V\setminus V^*$ to be nonempty and consider $f^*$ to be the solution of the following Dirichlet boundary problem
\begin{equation}
    \begin{cases}
        \inflap f^*(u)=0 \qquad \forall u\in \internalnodes\setminus V^*\\
        f^*(u)=d_{\boundary}(u) \qquad \forall u\in V^*\,.
    \end{cases}
\end{equation}
The existence and uniqueness of the solution to this problem is guaranteed from the discussion in \cite{manfredi2015nonlinear}.
Now we claim that $f^*\in \argmin_f \rayl_\infty(f)$ but $f^*\neq d_{\boundary}$, thus proving that the condition on the paths is a necessary condition to have uniqueness of the minimizer.

To prove the claim let $L=\max_{u\in \internalnodes\setminus V^*}\|\{\grad f^*\}(u)\|_{\infty}>0$. 
Then, necessarily from the condition $\inflap f(u)=0$, there exists a path $\gamma^*=\{v_1,\dots, v_m\}$ such that 
\begin{enumerate}
    \item $v_1, v_m\in V^*$,
    \item $v_i\in \internalnodes\setminus V^*$ for all $1<i<m$,
    \item  $\grad f^*(v_i,v_{i+1})=L$ for all $i=1,\dots m-1$.
\end{enumerate}  

  Note that $d_{\boundary}(v_m)\leq  d_{\boundary}(v_1)+d(v_1,v_m)\leq  d_{\boundary}(v_1)+\length(\gamma^*)$,
  Moreover, assume by contradiction that the last inequality was actually an equality. Then, consider the path $\gamma'$ obtained by concatenating the shortest path from the boundary to $v_1$ with $\gamma^*$ and then with the shortest path from $v_m$ to $u_0$. 
From the definition of $V^*$, $d_{\boundary}(u_0)=d(v_m,u_0)+d(v_m,\boundary)=d(v_1,u_0)+d(v_1,\boundary)$, which shows that $\gamma'$ is a shortest path from $u_0$ to the boundary, i.e.:
\begin{equation*}
\begin{aligned}
\length(\gamma')&=d_\boundary(u_0)=d(v_m,u_0)+\length(\gamma^*)+d_\boundary(v_1)\\
&=d(v_m,u_0)+d_\boundary(v_m)-d_\boundary(v_1)+d_\boundary(v_1)=d_\boundary(u_0).
\end{aligned}
\end{equation*} 
However, $\gamma'$ passes through nodes not included in $V^*$, yielding a contradiction.

Thus
  \begin{equation}\label{eq:11}
    d_{\boundary}(v_m)<  d_{\boundary}(v_1)+\length(\gamma^*),
  \end{equation}
yielding  
  \begin{equation*}
    \begin{aligned}
      L=\frac{f^*(v_m)-f^*(v_1)}{\length(\gamma^*)}=\frac{d_{\boundary}(v_m)-d_{\boundary}(v_1)}{\length(\gamma^*)}< 1
    \end{aligned}
  \end{equation*}

 It is easy to see that $f^*$ satisfies $\|\grad f^*\|_{\infty}=1$ and $\|f^*\|_{\infty}=\max_{v}d_{\boundary}(v)$. On the other hand, $\|\{\grad f^*\}(v)\|_{\infty}\leq L<1$ for all $u\in \internalnodes\setminus V^*$ showing that 
\begin{equation*}
    f^*\in \argmin_{g\in \mathcal{H}_0(\nodeset)}\rayl_{\infty}(g),
\end{equation*}
but $f^*\neq d_{\boundary}$.
Indeed it is easy to verify that the boundary distance function $d_{\boundary}$ satisfies the Eikonal equation 
\begin{equation*}
    \begin{cases}
        \|\{\grad f\}(u)\|_{\infty}=1 \qquad &\forall u\in \internalnodes,\\
        f(u)=0 \qquad &\forall u\in \boundary.
    \end{cases}
\end{equation*}
\end{proof}

We conclude this section with some illustrative examples where we investigate numerically the convergence of the $p$-Laplacian eigenpairs as $p$ goes to $\infty$.
\begin{figure}
  \begin{center}
    \begin{minipage}{.46\textwidth}    
      \includegraphics[width=\textwidth]{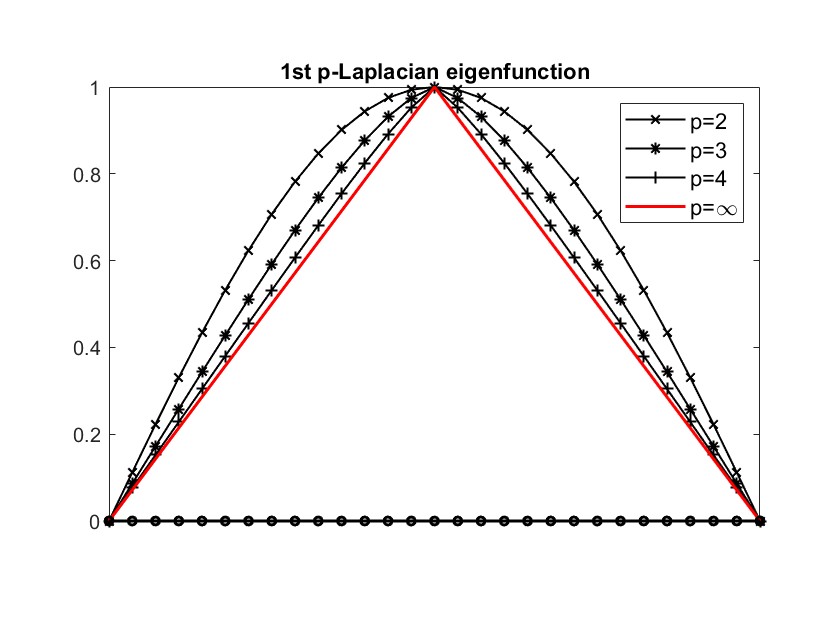}
      \includegraphics[width=\textwidth]{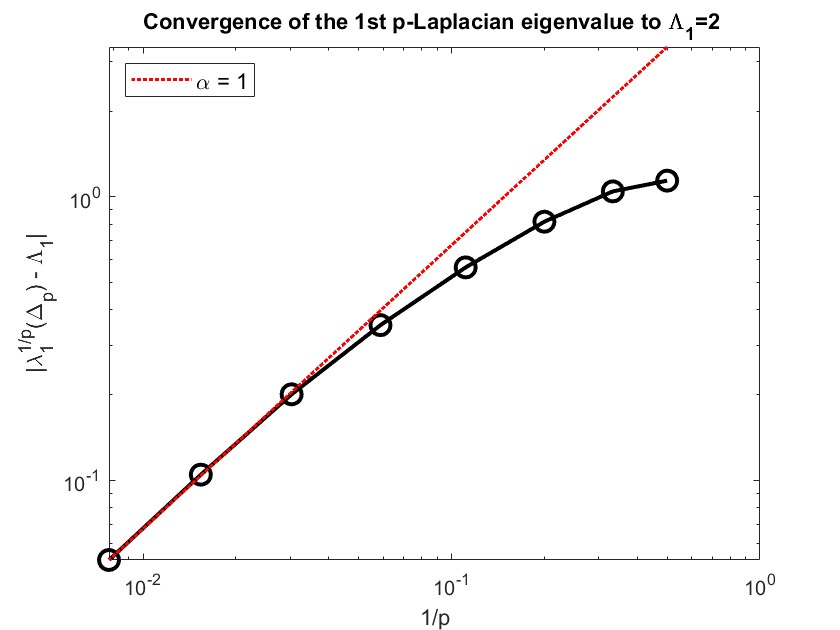}
      \includegraphics[width=\textwidth]{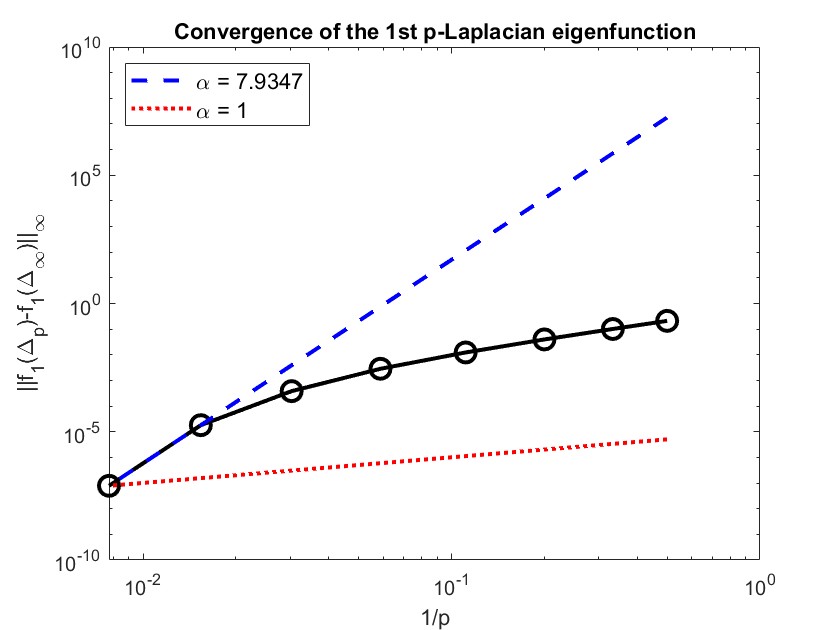}
    \end{minipage}
    \begin{minipage}{.46\textwidth}     
      \includegraphics[width=\textwidth]{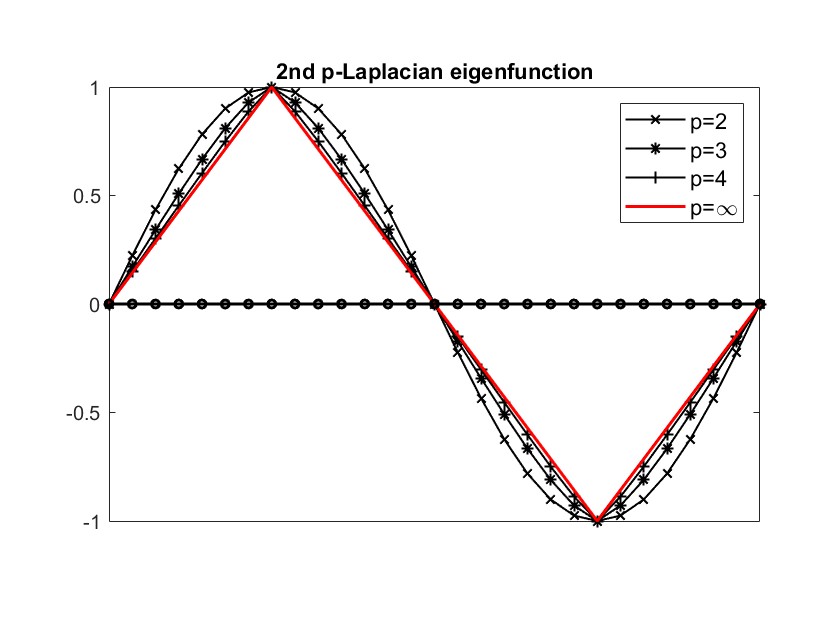}
      \includegraphics[width=\textwidth]{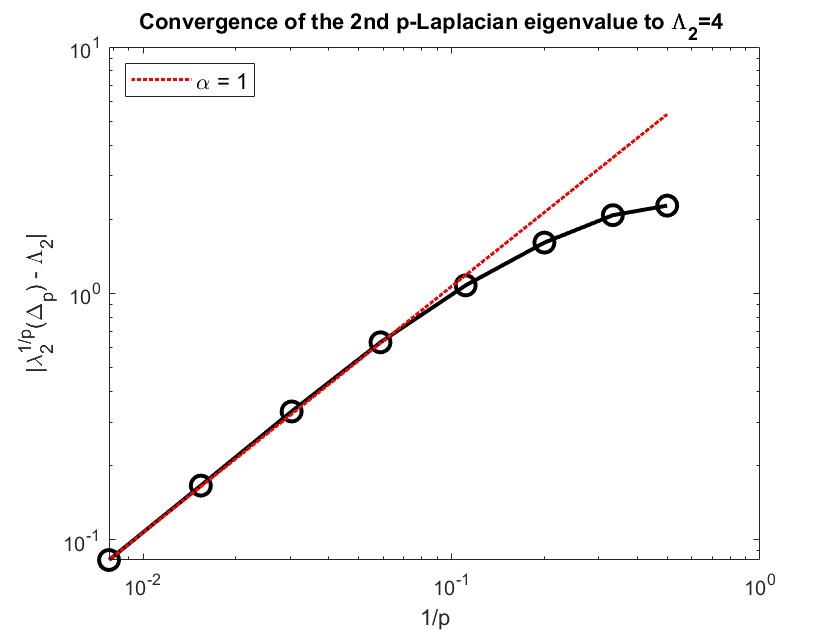}
     \includegraphics[width=\textwidth]{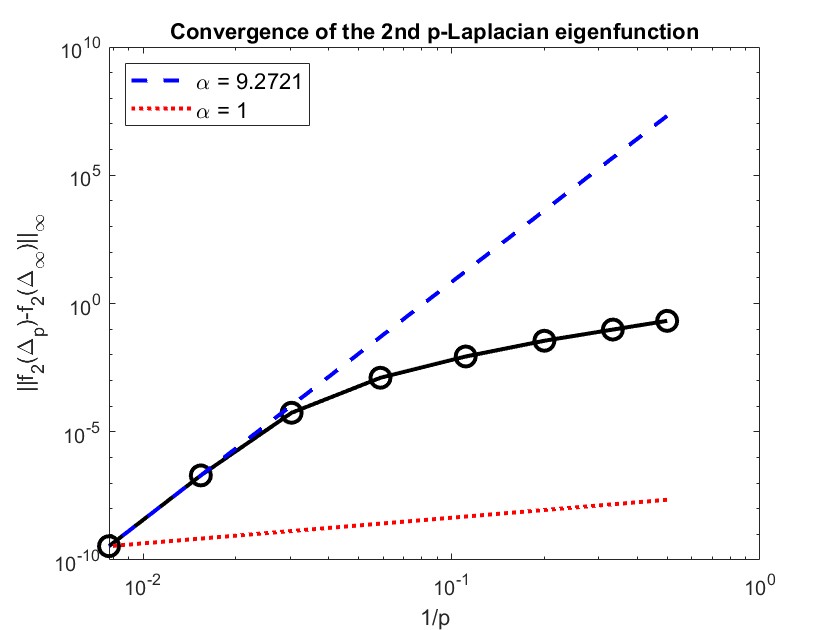}
    \end{minipage}
    \end{center}
\caption{Study of the convergence of the first two variational $p$-Laplacian eigenpairs to the $\infty$ ones on a line graph of 29 nodes with homogeneous Dirichlet boundary conditions on the extreme nodes.
First row left panel and right panel, respectively: interpolated representation of the first and second eigenfunction of the graph $p$-Laplacian for different values of $p$.  In second row, left and right panels, we report the absolute error between the $p$-root of the 1st and 2nd variational eigenvalues of the $p$-Laplacian $\lambda_1^{1/p}(\plap)$ ( $\lambda_2^{1/p}(\plap)$) and $\Lambda_1(\inflap)=2$ ($\Lambda_2(\inflap)=4$) as a function of $1/p$. In the third row, left and right panels, we report the errors between the  eigenfunctions (first and second) of the $p$-Laplacian and the corresponding infinite limit in \eqref{eq.12} as a function of $1/p$. Both the error plots of the $p$-eigenfunctions and $p$-eigenvalues are in loglog scale and are compared with the function $(1/p)^\alpha$ for different values of $\alpha$.}\label{Fig-Numerical-test_2}
\end{figure}

\begin{example}
  We consider two test cases. The first test case concerns a path graph of 29 equidistant nodes embedded in the interval $[0,1]$. The two extremal nodes are boundary nodes and the length of the edges is scaled according to the embedding in [0,1]. 
   In the second test case we address a grid of 15x15 equidistant nodes embedded in $[0,1]\times[0,1]$, with the edge lengths scaled accordingly to their length in the embedding. We set the nodes on the boundary of the square as the boundary of the graph.

  \Cref{Fig-Numerical-test_2} and \Cref{Fig-Numerical-test1} reports the results for the first and second test case respectively. We show the numerical convergence as $p\To\infty$ of the first, two and one respectively, variational eigenpairs (top row) and the estimated asymptotic convergence profiles with the estimated rates calculated as the slope of the least square line calculated on the last data points.

In the case of the line graph, we compute the first and second $p$-Laplacian eigenpairs for exponentially increasing values of $p$, precisely we set $p=2^l+1$ with $l=0,\dots,7$. In the second case of the grid we compute only the first eigenpair for the same values of $p$. The eigenpairs are computed using the method discussed in \cite{deidda2024_spec_energy}. 

We report illustrative representation of the computed eigenfunctions obtained by linear interpolation of the numerical values of the eigenfunctions at the nodes of the graph. Additionally, we report the errors between the computed eigenvalues and the limit $\infty$-Laplacian eigenvalues $\Lambda_1(\Delta_{\infty})$ and $\Lambda_2(\Delta_{\infty})$. Indeed we know that $\Lambda_1(\Delta_{\infty})=R_1^{-1}=2$, $\Lambda_1(\Delta_{\infty})=R_2^{-1}=4$ on the line graph and $\Lambda_1(\Delta_{\infty})=R_1^{-1}=2$ on the grid. 
Finally, we also include the error in $\infty$-norm between the computed eigenfunction and the expected limit eigenfunction, where all the eigenfunctions are normalized so to satisfy $\|f\|_{\infty}=1$. 
In particular, in the case of the line graph it can be easily observed that the two functions
\begin{equation}\label{eq.12}
    \begin{aligned}
        &f_1(\Delta_{\infty})(k)=f_1(\Delta_{\infty})(30-k)=(k-1)/14 \quad k=1,\dots, 15 \\
        &f_2(\Delta_{\infty})(k)=f_2(\Delta_{\infty})(16-k)=(k-1)/7 \quad k=1,\dots, 8 \\
        &f_2(\Delta_{\infty})(14+k)=f_2(\Delta_{\infty})(30-k)=-(k-1)/7 \quad k=1,\dots, 8
    \end{aligned}
\end{equation}
satisfy the $\infty$-limit eigenvalue equation in \eqref{limitin_inf_eigenvalue_eq} with eigenvalues $\Lambda_1=2$ and $\Lambda_2=4$, respectively.
On the other hand, in the case of the grid graph, the expected limit of the first eigenfunctions of the $p$-Laplacian is computed by solving numerically the $\infty$-limit eigenvalue equation \eqref{limitin_inf_eigenvalue_eq} with $\Lambda=2$. As suggested in \cite{BungertInfNumerical}, the limit eigenvalue equation is solved using the fixed point iteration:
\begin{equation*}
\begin{aligned}
    \tilde{f}^{(k)}=f^{(k-1)}-\delta\Theta_{\infty}(f^{(k-1)})\,, \\ f^{(k)}=\tilde{f}^{(k)}/\|\tilde{f}^{(k)}\|_{\infty},
\end{aligned}
\end{equation*}
where $\Theta_{\infty}(f)$ is defined in \eqref{limitin_inf_eigenvalue_eq} and $\delta$ is a fixed parameter smaller than $1$.
The starting point $f^{(0)}$ is chosen as the first eigenfunction of the $p$-Laplacian for some $p>2$, and convergence is considered achieved when $\Theta(f^{(k)})<\epsilon$ and $\|f^{(k)}-f^{(k-1)}\|_{\infty}<\epsilon$ for some small value of $\epsilon$.
The convergence rate is compared with the rate $1/p^{\alpha}$ for different values of $\alpha$, we recall that these appear as lines of slope $\alpha$ in loglog scale plottings.

In agreement with \eqref{eq_monotonocity_variational_eigenvalue}, all experiments suggest that the eigenvalues converge sublinearly as $O(1/p)$. On the contrary, we observe a higher convergence rate for the eigenfunctions. However, it is worth mentioning that, especially for large values of $p$, we have little control over the accuracy of the computed p-Laplacian eigenfunctions.
\begin{figure}
  \begin{center}
    \begin{minipage}{.4\textwidth}    
      \includegraphics[width=\textwidth]{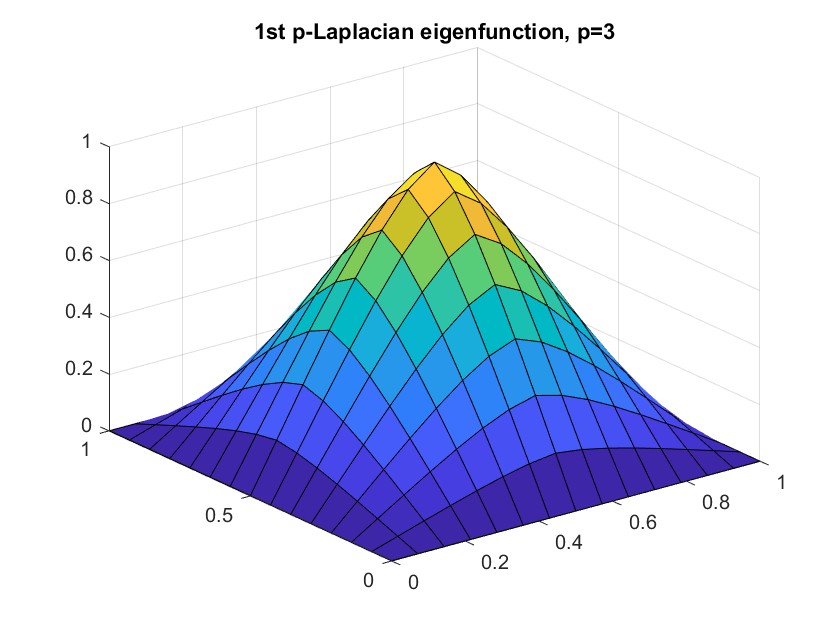}

      \includegraphics[width=\textwidth]{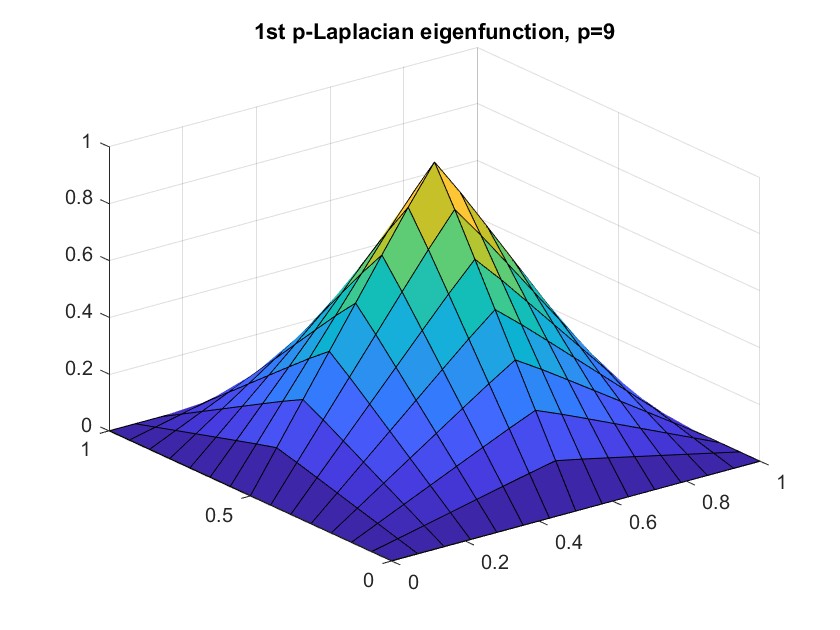}
    \end{minipage}
    \qquad
    \begin{minipage}{.4\textwidth}     
      \includegraphics[width=\textwidth]{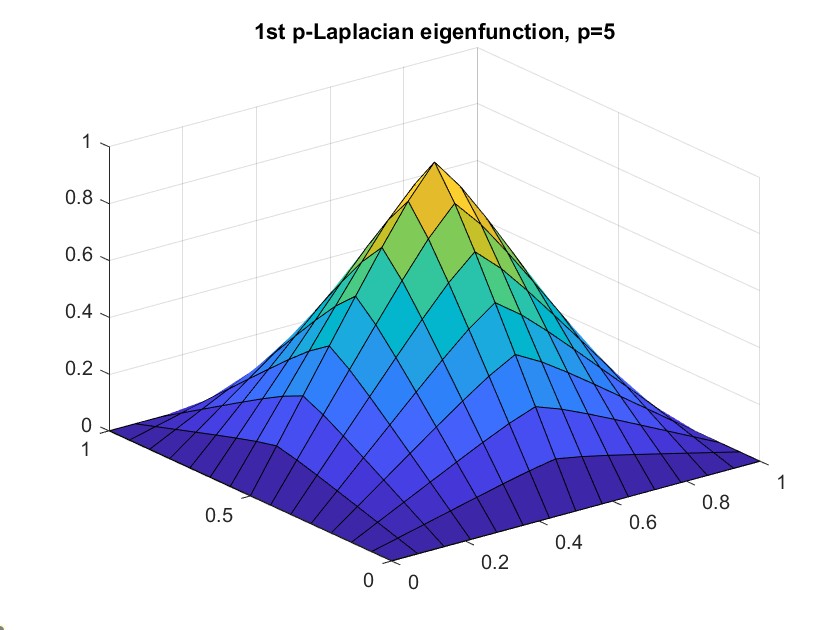}

      \includegraphics[width=\textwidth]{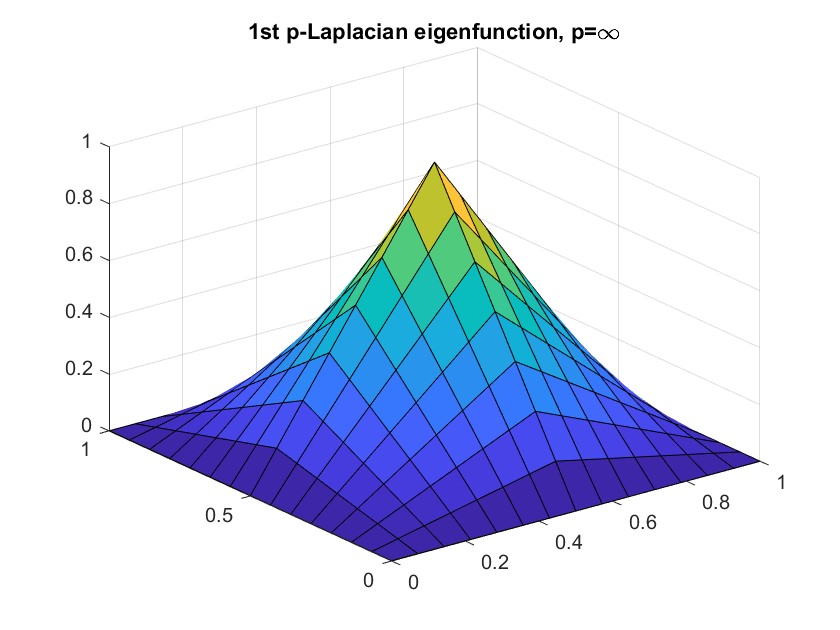}
    \end{minipage}\\[1em]
    \includegraphics[width=0.46\textwidth]{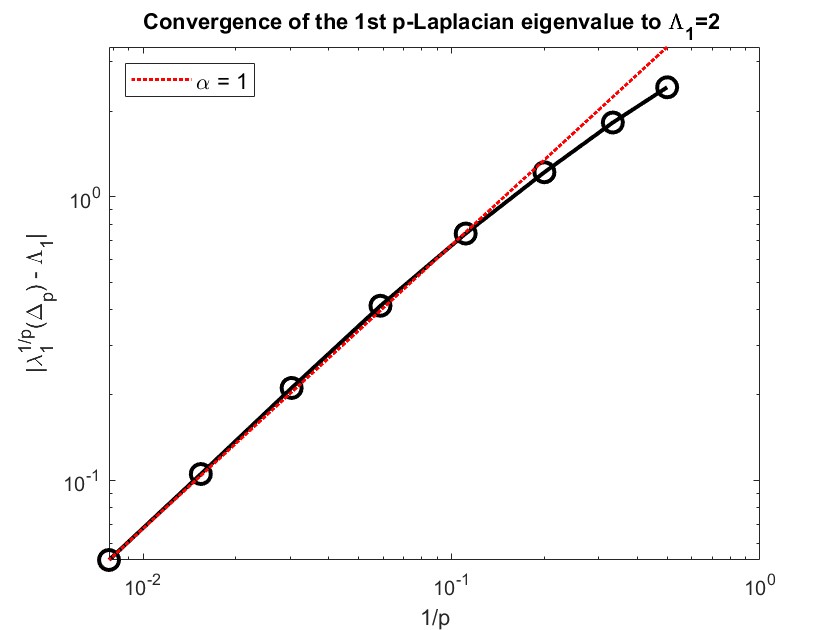}
    \includegraphics[width=0.46\textwidth]{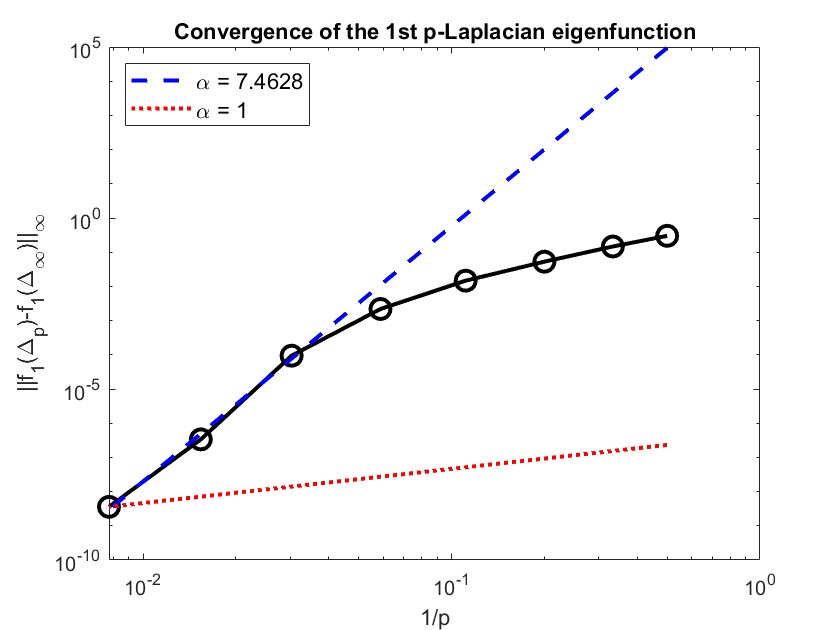}
    \end{center}
\caption{Study of the convergence of the first eigenpair of the $p$-Laplacian to the limit $\infty$-eigenpair.
First and second row from left to right and top to bottom: first eigenfunction of the graph $p$-Laplacian for  $p=3,5,9$. The last panel shows the expected limit eigenfunction obtained by solving the $\infty$-limit eigenvalue equation \eqref{limitin_inf_eigenvalue_eq} with $\Lambda=2$. The left panel in the third row reports the  absolute error between the $p$-th root of the 1st eigenvalue of the $p$-Laplacian $\lambda_1^{1/p}(\plap)$ and $\Lambda_1(\inflap)=2$. The right panel of third row reports the error between the 1st eigenfunction of the $p$-Laplacian and the expected limit eigenfunction. All the errors are plotted as functions of $1/p$. Both the error plots of the $p$-eigenfunctions and $p$-eigenvalues are in loglog scale and are compared with the function $(1/p)^\alpha$ for different values of $\alpha$.}
\label{Fig-Numerical-test1}
\end{figure}

\end{example}

\end{NEW}

%



\section{The generalized $\infty$-eigenvalue problem}\label{sec_inf_2}
In this section we discuss a different approach to the $\infty$-Laplacian eigenvalue problem. It is based on the definition of the generalized critical points of a non differentiable functional. The same approach has been deeply investiagted in the study of the graph $1$-Laplacian eigenpairs \cite{chang2016spectrum, hein2010inverse}. 
More in detail, in the first paragraph, we first recall the definition of the generalized eigenpairs and the corresponding $\infty$-variational eigenvalues. Then, by means of the subgradients of the convex functions $f\mapsto \|f\|_{\infty}, \|\grad f\|_{\infty}$, we establish an explicit algebraic equation satisfied by a generalized eigenpair.
In the subsequent paragraphs we prove a topological characterization of the geometrical $\infty$-eigenpairs similar to the one proved in Proposition \ref{optimal_paths_limiting_eigenfunctions}. This characterization allows us to compare the two different formulations of the $\infty$-eignepairs.
In particular, we prove that any solution of the equation in \cref{limitin_infty_eigenvalue_theorem} is a generalized eigenpair, while the generalized eigenpairs solve the limit equation in \cref{limitin_infty_eigenvalue_theorem} if we consider a proper subgraph of $\Gc$.
The generalized $\infty$-eigenpairs are formulated via the generalized ``critical points'' of the $\infty$-Rayleigh quotient given by: 
\begin{equation*}
  \rayl_{\infty}(f)=\frac{\|\grad f\|_{\infty}}{\|f\|_{\infty}}\,.
\end{equation*}
Indeeed, the $p$-Laplacian eigenpairs can be seen as critical points and values of the functional $\|\grad f\|_p$ on the manifold $\|f\|_p=1$, i.e. the points in which the differential of $\|\grad f\|_p$ is normal to $S_p=\{f:\internalnodes\rightarrow \R|\; \|f\|_p=1\}$.
Letting $p\rightarrow\infty$, we observe immediately that $\|\grad f\|_{\infty}$ is not a differentiable operator and $S_{\infty}$ is not a smooth manifold. Nevertheless, following~\cite{chang2021nonsmooth}, since $\|\grad f\|_{\infty}$ and $\|f\|_{\infty}$ are convex functions, we can generalize the notion of critical points as follows:
\begin{definition}\label{Generalized_eigenpair}
  We say that $(f,\Lambda)$ is a generalized $\infty$-eigenpair if and only if:
  \begin{equation*}
    0\in\partial\|\grad f\|_{\infty}\cap\Lambda\partial\|f\|_{\infty}\,.
  \end{equation*}
  where $\partial\|\grad f\|_{\infty}$ and $\partial\|f\|_{\infty}$ are the subgradients of the functions $(f\mapsto\|\grad f\|_{\infty})$ and $(f\mapsto\|f\|_{\infty})$, respectively.
\end{definition}
From~\cite{chang2021nonsmooth, deidda2025nonlinear}, we observe that Definition~\ref{Generalized_eigenpair} can be considered as the generalized critical point equation for the functional $(f\mapsto\|\grad f\|_{\infty})$ on $S_{\infty}$ since $\partial\|\grad f\|_{\infty}$ is a generalization of the differential of $\|\grad f\|_p$ when $p=\infty$, while, from Lemma 4.2 and 4.3 of \cite{chang2021nonsmooth}, (see also \cite{rockafellar2015convex}) the external cone to $S_{\infty}$ in a point $f$, i.e.,
\begin{equation*}
  C_{Ext}(f)=\{\xi:\nodeset\rightarrow\R|\:
     \SCAL{\xi}{g-f} \leq 0\;\forall g\in S_{\infty}\}
\end{equation*}
can be characterized as  
\begin{equation*}
  C_{Ext}(f)=\mathop{\cup}_{\lambda\geq 0}\lambda\partial\|f\|_{\infty}\,.
\end{equation*}
Moreover we point out that, from Theorem 5.8 of \cite{chang2021nonsmooth}, $\Lambda_k(\Delta_\infty)$ defined in \eqref{def_inf_var_eigenvalues} corresponds to the $k$-th Krasnoselskii $\infty$-variational eigenvalue, and it is a generalized critical value of $\rayl_{\infty}$ as in \cref{Generalized_eigenpair}.

Similar approaches have been used to study the $1$-Laplacian eigenpairs \cite{chang2016spectrum, hein2010inverse}, and more recently to study minimizers of $\|\grad f\|_{\infty}$ in $L^2$ and $L^\infty$ spaces, \cite{BungertInfinityL2, bungert2021eigenvalue}.
\new{ In particular, from \eqref{eq_monotonocity_variational_eigenvalue} and \cite{deidda2025nonlinear} we have that the $\infty$-variational eigenvalues are equal to the limit of the $p$-Laplacian variational eigenvalues as $p$ goes to $\infty$.
 Moreover, \cref{thm:packing-radii} and \cref{Variational_eienvalues_radius_estimates} can be reformulated in terms of inequalities relating the infinity variational eigenvalues and the packing radii of the graph as in the following Theorem.}
\begin{theorem}\label{Clarke_variational_eigen_characterization}
  Let $\{\Lambda_k\}$ be the Krasnoselskii variational eigenvalues and $R_k$ be the $k$-th packing radius defined in Definition \ref{k-th_radius}.
  Then,
  \begin{equation*}
    \Lambda_k\leq \frac{1}{R_k}\qquad\forall k=1,\dots,|\internalnodes|
  \end{equation*}
  and equality holds for $k=1,2$.
\end{theorem}
\begin{proof} 
\new{The proof follows trivially from \cref{thm:packing-radii} and \cref{Variational_eienvalues_radius_estimates}. However, the theorem can also be proved without necessarily using the fact that the $\infty$-variational eigenvalues are the limit of the $p$-Laplacian variational ones. For completeness we incluse also this alternative proof. First of all, observe that the inequality can be proved just following the proof of Prop \ref{Variational_eienvalues_radius_estimates}.}
 I.e, consider $u_1,\dots,u_k$ such that $d(u_i,u_j)\geq 2r, \; d(u_i,B)\geq r\;\: \forall i,j=1,\dots,k$ and define the $k$ linearly independent functions
  \begin{equation*}
    f_j(u)=\max\lbrace R_k-d(u,u_j),0\rbrace\,, \quad j=1,\dots,k\,.
  \end{equation*}
  Then the set $A_k=span\lbrace f_j \rbrace_{j=1}^k$ has Krasnoselskii genus equal to $k$, $\gamma(A_k)=k$, and thus:
  \begin{equation*}
    \Lambda_k\leq \max_{f\in A_k}\rayl_{\infty}(f)\,,
  \end{equation*}
  from which the subsequent inequality readily follows $\max_{f\in A_k}\rayl_{\infty}(f)\leq 1/R_k$.

  \new{ Similarly, also equalities $\Lambda_1= R_1^{-1}$ and $\Lambda_2= R_2^{-1}$ can be proved without necessarily using the properties of the limit $p$-Laplacian eigenpairs. Indeed, because of Corollary \ref{First_infinity_eigenvalue}, we immediately have $\Lambda_1= R_1^{-1}$. Moreover, any $A\in\mathcal{F}_2(S_{\infty})$ necessarily contains a symmetric, closed, and connected subset of $S_\infty$.}   
  Then, given the function 
$\psi\in C(S_{\infty}, \R)$ mapping any function $f$ to $\|f^+\|_{\infty}-\|f^-\|_{\infty}$, we have that
  for any $A\in\mathcal{F}_2(S_{\infty})$ there esists $f_A\in A$ such that $\psi(f_A)=0$. In other words, there exist $u^+,u^-\in\nodeset$ such that 
   $\|f\|_{\infty}=f_A(u^+)=-f_A(u^-)$.
 To conclude, from \eqref{Lipschitzianity}, we observe that 
  \begin{equation*}
   d(u^\pm,\boundary)\geq \frac{\|f_A\|_{\infty}}{\|\grad f_A\|_{\infty}}\,,
    \qquad 
    d(u^+,u^-)\geq 2\frac{\|f_A\|_{\infty}}{\|\grad f_A\|_{\infty}}\,,
  \end{equation*}
  i.e. $\|f_A\|_{\infty}/\|\grad f_A\|_{\infty}\leq R_2$, yielding $\Lambda_2^\geq \min_{A\in\mathcal{F}_2} \|\grad f_A\|_{\infty}/\|f_A\|_{\infty}\geq 1/R_2\,.$
\end{proof}

\begin{NEW}
\begin{remark}
We point out that generally the number of nodal domains of a generalized variational eigefunction $f_k$ can not be used to establish a lower bound for the corresponding eigenvalue in terms of some packing radius as in \cref{cor:lower-upper}. This is due to the fact that generalized eigenfunctions are “less"-rigid than solutions of the $\infty$-limit eigenvalue equation and so they can oscillate arbitrarily far from the nodes and edges of the graph where $|\grad f|$ and $|f|$ are maximal. We refer to \cite{deidda2025nonlinear} and, in particular to Example 6.12 therein for more details. However a lower bound like the one in \cref{cor:lower-upper} can be recovered by considering the number of perfect nodal domains induced by $f_k$, where a perfect nodal domain is a nodal domain where some node $u$ exists such that $|f(u)|=\|f\|_{\infty}$, see Theorem 8.6 in \cite{deidda2025nonlinear}.
\end{remark}
\end{NEW}


Next, we study the structure of the two sets $\partial\|\grad f\|_{\infty}$ and $\partial\|f\|_{\infty}$. 
As in remark \ref{Remark_boundary_conditions}, we think about both $f\mapsto\|f\|_{\infty}$ and $f\mapsto\|\grad f\|_{\infty}$ as functions from $\mathcal{H}(\internalnodes)$ to $\R$. 
Then, as done in \cite{BungertInfinityL2}, because of homogeneity, it is easy to derive the following characterization, (see  also \cite{burger2016spectral, deidda2025nonlinear}):
\begin{equation*}
    \partial\big(f\mapsto\|f\|_{\infty}\big)=\{\xi:\internalnodes\to \R\,|\, \|g\|_{\infty}\geq \SCAL{\xi}{g}\,\forall g:\internalnodes\to\R,\;\: \|f\|_{\infty}= \SCAL{\xi}{f}  \}\,,
\end{equation*}
i.e.
\begin{equation}\label{Subdifferetnial_characterization_2}
  \partial\|f\|_{\infty}:=\left\{ \xi \;\bigg|
    \begin{array}{lr}
      \|\xi\|_{1,\nodeset}=1,\; \xi(u)=0\;\;
      \forall u\in\internalnodes\setminus\nodemaxset(f)\,,\\[.4em] 
      |\,\xi(u)|\, |f(u)|=\xi(u) f(u)\;\; \forall u\in\nodemaxset(f)
    \end{array}
  \right\}\,.
\end{equation}
Moreover, we can use the subdifferential chain rule for linear transformations $\Big(\partial \big(x\mapsto \phi(Ax)\big)=A^T\partial \big(y\mapsto \phi(y)\big)|_{y=Ax}\Big)$ to characterize $\partial(f\mapsto \|\grad f\|_{\infty})$~\cite{rockafellar2015convex}:
\begin{equation}\label{Subdifferetnial_characterization_1}
  \begin{aligned}
    \partial\|\grad f\|_{\infty}:=\left\{ -\divg \Xi\, \bigg|\!\!
    \begin{array}{lr}
      \|\Xi\|_{1,\edgeset}=1,\; \Xi(u,v)=0\;\;
      \forall(u,v)\in \edgeset\setminus\edgemaxset(f)\,,\\[.4em]
      |\,\Xi(u,v)\,| |\grad f(u,v)|=\Xi(u,v)\grad f(u,v)\;
      \forall (u,v)\in\edgemaxset(f)
    \end{array}
    \!\!\!\right\}
  \end{aligned}
\end{equation}
where the divergence operator in defined as in Remark \ref{Remark_boundary_conditions},
and we used the edge and node norms defined in~\eqref{eq:norms}.
Then we can give the following definition of generalized $\infty$-eigenpair.
\begin{definition}\label{Generealized_eigenpair_equation}
  $(f,\Lambda)$ is a generalized $\infty$-eigenpair if and only if there exist $\xi\in\partial\|f\|_{\infty}$ and $\Xi$ with $-\divg(\Xi)\in\partial\|\grad f\|_{\infty}$ such that
  \begin{equation*}
    -\divg\,\Xi (u)=\Lambda \xi(u) \quad \forall u\in \internalnodes\,. 
  \end{equation*}
\end{definition}
Note that, putting together the above definition of generalized eigenpair and the characterization of the subgradients $\partial\|f\|_{\infty}$  and $\partial\|\grad f\|_{\infty}$ (eqs.~\eqref{Subdifferetnial_characterization_2} and~\eqref{Subdifferetnial_characterization_1}), we have the following Proposition, whose proof is immediate.
\begin{proposition}\label{Prop_eq_characterization_of_gen_inf_eigenpairs}
  $(f,\Lambda)$ with $f\in\mathcal{H}_{0}(\nodeset)$, is a generalized $\infty$-eigenpair if and only if there exist $\xi:\internalnodes\rightarrow\R$ and $\Xi:\edgeset\rightarrow\R$ such that
  \begin{equation}\label{Generealized_eigenpair_system}
    \begin{cases}
      -\divg \Xi (u)=\Lambda \xi(u) \quad &\forall u\in\internalnodes\\ 
      \|\Xi\|_{1,\edgeset}=1\\ 
      \|\xi\|_{1,\nodeset}=1\\
      |f(u)|=\|f\|_{\infty}\quad &if\;\xi(u)\neq 0\\
      |\grad f(u,v)|=\|\grad f\|_{\infty}\quad &if\;\Xi(u,v)\neq 0\\
      \mathrm{sign}\big(\Xi(u,v)\big)=\mathrm{sign}\big(\grad f(u,v)\big)
      \qquad &if\;\Xi(u,v)\neq 0\\
      \mathrm{sign\big(\xi(u)\big)}=\mathrm{sign\big(f(u)\big)}
      \quad &if\;\xi(u)\neq 0
    \end{cases}\,.
  \end{equation}
  Moreover, up to redefining $\Xi(u,v)=\Big(\Xi(u,v)-\Xi(v,u)\Big)/2$, we can assume $\Xi(u,v)=-\Xi(v,u)$.
\end{proposition}
\begin{remark}
We would like to observe that since $\rayl_{\infty}$ is a locally Lipschitz function of $\R^n\setminus\{0\}$, the notion of critical point can also be generalized considering the Clarke subderivative  $\partial^{Cl}\rayl_{\infty}$, see \cite{clarke1990optimization}, i.e. $f$ is a Clarke $\infty$-eigenpairs iff
\begin{equation}\label{Clarke_generalized_critical_point_Eq}
  0\in\partial^{Cl}\rayl_{\infty}(f)\,.
\end{equation}
From \cite{clarke1990optimization, hein2010inverse} we know that:
\begin{equation}\label{Subgrad_inclusion}
  \partial^{Cl}\rayl_{\infty}(f)\subseteq
  \frac{\partial\|\grad f\|_{\infty}\|f\|_{\infty}-\partial\|f\|_{\infty}\|\grad f\|_{\infty}}%
  {\|f\|^2_{\infty}}\,,
\end{equation}
which shows that the notion of generalized $\infty$-eigenpair (\Cref{Generealized_eigenpair_equation}) extends the notion of Clarke-eigenpair. We refer to the work in \cite{zhang2021homological} for an example where the set of the Clarke eigenvalues is strictly included in the set of the generalized eigenvalues.
Nevertheless, \Cref{Generealized_eigenpair_equation} provides some practical advantages with respect to the eq.~\eqref{Clarke_generalized_critical_point_Eq} since, differently from $\partial^{Cl}\rayl_{\infty}(f)$, both $\partial\|f\|_{\infty}$ and $\partial\|\grad f\|_{\infty}$ can be explicitly identified.
\end{remark}

\begin{NEW}
    From the characterization of the subgradients \eqref{Subdifferetnial_characterization_1} and \eqref{Subdifferetnial_characterization_2}, and the reformulation of the generalized $\infty$-eigenvalue problem in \cref{Generealized_eigenpair_system}, it is easy to see that the $\infty$-eigenpairs are usually not unique. Indeed, given $(f,\Lambda)$ a generalized  $\infty$-eigenpair, if there exists a node $u$ such that $|f(u)|< \|f\|_{\infty}$ and $\|\{\grad f\}(u)\|_{\infty}< \|\grad f\|_{\infty}$. Then we can consider a continuous family of perturbations $f_{\epsilon}$ of $f$ on the only node $u$ (e.g. $f_\epsilon(v)=f(v)+\epsilon \delta_u(v)$) such that, as long as $f_\epsilon$ satisfies 
    \begin{equation}
    |f_{\epsilon}(u)|< \|f\|_{\infty} \quad \text{and} \quad \|\{\grad f_{\epsilon}\}(u)\|_{\infty}\leq \|\grad f\|_{\infty},
    \end{equation}
    $f_\epsilon$ is a generalized $\infty$-eigenfunction of the same eigenvalue $\Lambda$. As a particular case, any minimizer of the $\infty$-Rayleigh quotient, that we have observed in \cref{ex:2} and \cref{thm_distance_unique_minimizer}  generically to be not unique, is a generalized $\infty$-eigenpair.   
    It is clear that, since $\rayl_p(f)$ converges to $\rayl_{\infty}(f)$, any limit of $p$-Laplacian eigenpairs for $p\rightarrow\infty$ is always a generalized $\infty$-eigenpair. However, the same is not straightfoward for the solutions of the $\infty$-limit eigenvalue equation \eqref{limitin_inf_eigenvalue_eq}. Indeed, as we have observed in \cref{ex:1}, there exist solutions of the  $\infty$-limit eigenvalue equation that are not limit of $p$-Laplacian eigenpairs. In the next section we prove that any solution of the $\infty$-limit eigenvalue equation is also a generalized $\infty$-eigenpair. To this end we prove that the generalized $\infty$-eigenpairs are uniquely characterized by a topological property analogue to the one satisfied by the solutions of the $\infty$-limit eigenvalue equation, see \cref{optimal_paths_limiting_eigenfunctions} and \cref{Infinity_eigenfunctions_characterization}.
\begin{remark}
    We dedicate this remark to highlight analogies and differences of the $\infty$-Laplacian eigenvalue problem and the theory of Lipschitz learning in $\ell_p$-semi-supervised learning \cite{calder2018game, calder2019consistency, slepcev2019analysis,  roith2023continuum, kyng2015algorithms, el2016asymptotic, flores2022analysis}. 
    In this case one wants to find a solution to the problem 
\begin{equation}\label{l_p_learning_problem}
    \min\{\|\grad f\|_{p}|\; f(u)=g(u) \quad \forall u\in \boundary\},
\end{equation}
 where $g$ is a fixed boundary condition. When $p\in(1,\infty)$, the minimum is unique and satisfies the system of equations
 \begin{equation}\label{p-lap_learning}
 \begin{cases}
  \plap f(u)=0 \qquad &\forall u\in \internalnodes,\\
  f(u)=g(u) \qquad &\forall u\in \boundary.
 \end{cases}
\end{equation}
 However, when one considers the case of Lipschitz learning, i.e. $p=\infty$, the minimum is not unique anymore. Uniqueness is recovered by considering the absolutely minimal Lipschitz extension of $g$, which corresponds to the solution of the system 
  \begin{equation}\label{inf_lap-learning}
  \begin{cases}
  \inflap f(u)=0 \qquad &\forall u\in \internalnodes,\\
  f(u)=g(u) \qquad &\forall u\in \boundary.
  \end{cases}
\end{equation}
 In particular, the unique solution of \eqref{inf_lap-learning} is also the limit of the solutions of \eqref{p-lap_learning} as $p$ goes to $\infty$.

 Similarly to \eqref{l_p_learning_problem}, the generalized eigenpairs simply correspond to the “critical points" of the $\infty$-Rayleight quotient and they are generally not unique. A smaller subset of $\infty$-eigenpairs can always be obtained by looking at the limit of $p$-Laplacian eigenpairs, i.e. critical points of the $p$-Rayleigh quotient, as $p$ goes to $\infty$. So, as done in going from \eqref{p-lap_learning} to \eqref{inf_lap-learning}, we have identified a limit equation satisfied by the limit of $p$-eigenpairs (see \eqref{limitin_inf_eigenvalue_eq}). However, differently from the case of $\ell_p$ semi-supervised learning, the solutions of the limit eigenvalue equation \eqref{limitin_inf_eigenvalue_eq}, are not all and only the limits of $p$-Laplacian eigenpairs (see \cref{ex:1}). 
\end{remark}

\end{NEW}



\subsection{Topological characterization}
In this section we discuss a geometrical characterization of the generalized $\infty$-eigenfunctions similar to the one proved in Proposition \ref{optimal_paths_limiting_eigenfunctions}. 
In both these characterizations, given an eigenpair, $(f,\Lambda)$, we prove the existence of “good" paths that connect points in $\nodemaxset(f)\cup\boundary$ and whose length matches the value of the eigenvalue. 
However, differently from Proposition \ref{optimal_paths_limiting_eigenfunctions}, where we proved the existence of a “good" path $\Gamma$ for any extremal point $v\in\nodemaxset(f)$, in the case of generalized $\infty$-critical points, extremal points that do not correspond to any “good" path may indeed exist. 
Moreover, contrary to the case of limit $\infty$-eigenpairs where the existence of such “good" paths was only a necessary condition for the existence of an eigenpair, in the case of generalized eigenpairs, the existence of “good" paths is also a sufficient condition. 
\begin{proposition}\label{Infinity_eigenfunctions_characterization}
  The pair $(f,\Lambda)$ is a generalized $\infty$-eigenpair with $f$ not a constant function if and only if there exists a path $\Gamma=\{(u_i,u_{i+1})\}_{i=1}^{n-1}$ such that:
  \begin{enumerate}
  \item $u_i\in \nodemaxset(f)\cup \boundary\quad  i=1,n$;
  \item $(u_i,u_{i+1})\in\edgemaxset(f)$ and $f(u_i)>f(u_{i+1})$ for all $i=1,\dots,n$;
  \item assuming without loss of generality that $f(u_1)>0$ , if  
    $u_n\in\boundary$ then $\Lambda=\frac{1}{\length(\Gamma)}$, if $u_n\notin\boundary$ then $f(u_n)=-f(u_1)$ and $\Lambda=\frac{2}{\length(\Gamma)}$.
    Moreover
    \begin{equation*}
      \frac{1}{\Lambda}
      =\min\bigg\{\min_{\{v|f(v)=-\|f\|_{\infty}\}}
      \frac{d(u_1,v)}{2},d_{\boundary}(u_1)\bigg\}\,.
    \end{equation*}
  \end{enumerate}
\end{proposition}
\begin{proof}
Assume that $(f,\Lambda)$ is a generalized $\infty$-eigenpair and let $(\xi,\Xi)$ be the two subgradients as in Definition \eqref{Generealized_eigenpair_system}, i.e., such that
\begin{equation}\label{Prop_characterization_eq_1}
    -\divg(\Xi)=\Lambda \xi\,,
\end{equation}
and $\Xi(u,v)=-\Xi(v,u)$. Let $\xi(u_1)\neq 0$ and without loss of generality assume $f(u_1)>0$. From eq.~\eqref{Generealized_eigenpair_system} there has to exist an edge $(u_1,u_2)\in\edgemaxset(f)$ such that $\Xi(u_1,u_2)\neq 0$, i.e. $f(u_2)=f(u_1)-\|\grad f\|_{\infty}/\edgelength_{u_1u_2}<f(u_1)$.
Let us focus now on the second node $u_2$. Since $f(u_2)<f(u_1)$, if $u_2\notin\boundary$ and $f(u_2)\neq -f(u_1)$, we have that $\xi(u_2)=0$. 
Moreover $\Xi(u_1,u_2)<0$ and $\Xi$ has to satisfy \eqref{Prop_characterization_eq_1}, i.e.:
\begin{equation*}
    -\divg(\Xi)(u_2)=\sum_{v\sim u_2}\edgelength_{vu_2}\Xi(v,u_2)=\Lambda \xi(u_2)=0\,,
\end{equation*}
hence, there must exist an edge $(u_3,u_2)\in\edgemaxset(f)$ such that $\Xi(u_2,u_3)> 0\,,$ i.e., $f(u_3)=f(u_2)-\|\grad f\|_{\infty}/\edgelength_{u_2 u_3}<f(u_2)$\,.
Iterating the above argument, we can define a path $\Gamma=\{(u_i,u_{i+1})\}_{i=1}^{n-1}$ such that $u_n\in\boundary\cup\nodemaxset(f)$, $(u_i,u_{i+1})\in\edgemaxset(f)$ and $f(u_i)>f(u_{i+1})$.
Furthermore, as in the proof of Proposition~\ref{optimal_paths_limiting_eigenfunctions}, it is easy to see that, given a function $f$ and a path $\Gamma$ that satisfies the first two items of the thesis, necessarily we have that, if $u_n\in\boundary$, $\Lambda=1/\length(\Gamma)$. On the other hand, if $f(u_n)=-f(u_1)$, $\Lambda=2/\length(\Gamma)$ and 
\begin{equation*}
  \frac{1}{\Lambda}=
  \min\bigg\{\min_{f(v)=-\|f\|_{\infty}}\frac{d(u_0,v)}{2},d_{\boundary}(u_0)\bigg\}\,.
\end{equation*}
Indeeed, proceeding by contraddiction, if one of the last equalities was not true then an edge $(v_1,v_2)$ with $|\grad f(v_1,v_2)|>\|\grad f\|_{\infty}$ would exist.
To prove the opposite inclusion, we assume that, given $(f,\Lambda)$, there exists a path  $\Gamma=\{(u_i,u_{i+1})\}_{i=1}^{n-1}$ with $f(u_1)>0$ that satisfies the thesis and look at the case $u_n\in\{v\,|\,f(v)=-f(u_1)\}$ with  $\Lambda=2/d(u_1,u_n)=2/\length(\Gamma)<d_{\boundary}(u_1)$ (the other case can be proved analogously). 
We aim at defining two functions $\Xi$ and $\xi$ that satisfy the sufficient conditions in Proposition \ref{Prop_eq_characterization_of_gen_inf_eigenpairs}.
Define the two functions as follows:
\begin{equation*}
    \begin{aligned}
      &\Xi(u,v):=
        \frac{\delta_{\Gamma}(u,v)}{\edgelength_{uv}}\:
        \frac{   \mathrm{sign}\big(\grad f(u,v)\big)}{\length(\Gamma)}
        \\[.5em]
      &\xi(v):=
        \frac{\delta_{u_1}(v)\mathrm{sign}\big(f(u_1)\big)
        +\delta_{u_n}(v)\mathrm{sign}\big(f(u_n)\big)}{2}\,,
    \end{aligned}
\end{equation*}
where $\delta_{\Gamma}$ $\delta_{u_i}$ are delta functions, i.e.,  $\delta_{\Gamma}(u,v)=1$ if $(u,v)\in \Gamma$ and $\delta_{\Gamma}(u,v)=0$ otherwise, with analogous definition for $\delta_{u_i}$.
Then, since $\xi$ and $-\divg \Xi$ belong respectively to the two subgradients $\partial\|f\|_{\infty}$ and $\partial\|\grad f\|_{\infty}$, an easy calculations shows that for any node $v\not\in\Gamma$ as well as for any node $v\in\Gamma$, with $v\neq u_1$ and $v\neq u_n$, we have:
\begin{equation*}
    -\divg \Xi(v)=0=\Lambda\, \xi(v)\,.
\end{equation*}
Finally, in the case $v=u_1$ (or analogously $v=u_n$), we have
\begin{equation*}
  -\divg \Xi(u_1)=
  \frac{\edgelength_{u_1 u_2}}{2}
  \Big(\Xi(u_2,u_1)-\Xi(u_1,u_2)\Big)
  =\frac{\mathrm{sign}\big((\grad f(u_2,u_1)\big)}{\length(\Gamma)}
  =\frac{\Lambda}{2}=\Lambda \xi(u_1)\,,
\end{equation*}
where we have used the fact that $\mathrm{sign}\big(\grad f(u_2,u_1)\big)=\mathrm{sign}\big(f(u_1)-f(u_2)\big)=1$, which follows by the hypothesis $f(u_1)>f(u_2)$.
This concludes the proof.
\end{proof}

As a corollary of Proposition \ref{Infinity_eigenfunctions_characterization}
and Proposition \ref{optimal_paths_limiting_eigenfunctions} it is easy to prove that any eigenpair that satisfies the limit eigenvalue equation~\eqref{limitin_inf_eigenvalue_eq} is also a generalized $\infty$-eigenpair.
\begin{corollary}\label{solutions_of_inf_eigenvalue_eq_are_generalized_inf_eigenpairs}
Let $(f,\Lambda)$ satisfy the limit eigenvalue equation~\eqref{limitin_inf_eigenvalue_eq} with $f$ not constant, then $(f,\Lambda)$ is also a generalized $\infty$-eigenpair according to Definition~\ref{Generealized_eigenpair_equation}. 
\end{corollary}
\begin{proof}
The proof is a consequence of the fact that, from Propositon \ref{optimal_paths_limiting_eigenfunctions}, for any eigenpair that satisfies the limit eigenvalue equation~\eqref{limitin_inf_eigenvalue_eq}, there exists a path $\Gamma$ that satisfies the hypotheses of Proposition \ref{Infinity_eigenfunctions_characterization}.
\end{proof}

Next we deal with the opposite problem and ask ourselves if any generalized $\infty$-eigenvalue (see Definition \ref{Generalized_eigenpair}) can be associated to an eigenfunction that solves \eqref{limitin_inf_eigenvalue_eq}.
As we will see in Example \ref{Ex_Generalized}, the answer is negative in general. Nevertheless, in Theorem \ref{formulations_comparison} we will see that the statement can always be proved to hold up to a subgraph of $\Gc$.
Before we prove this statement and discuss the properties of the subgraph, we observe that the $\infty$-eigenvalue problem can be reformualted in terms of a constrained weighted linear Laplacian eigenvalue problem.
Indeed, from the characterizations of the subgradient in~\eqref{Subdifferetnial_characterization_1} and~\eqref{Subdifferetnial_characterization_2}, it is possible to reformulate system~\eqref{Generealized_eigenpair_system} as in the following proposition.
\begin{proposition}\label{weighted_laplacian_equivalence}
  The pair $(f,\Lambda)$ is a generalized $\infty$-eigenpair if and only if there exist two admissible densities $\nodeweight:\internalnodes\rightarrow\R_+$, and $\edgeweight:\edgeset\rightarrow\R_+$, with $\edgeweight_{uv}=\edgeweight_{vu}$ such that:
  \begin{equation}\label{eq-weighted_laplacian_equivalence}
    \begin{cases}
      -\divg(\edgeweight\odot\grad f)(u)=\Lambda \nodeweight_u f(u) \quad &\forall u\in\internalnodes\\
      |\grad f(u,v)|=\|\grad f(u,v)\|_{\infty}\quad &\text{if}\;\edgeweight_{uv}>0 \\
      |f(u)|=\|f(u)\|_{\infty}\quad &\text{if}\;\nodeweight_u>0\\
      \|\edgeweight \odot\grad f\|_{1,\edgeset}=1\\
      \|\nodeweight \odot f\|_{1,\nodeset}=1\\
      f(u)=0 \quad &if\;u\in\boundary
    \end{cases}\,.
  \end{equation}
\end{proposition}
\begin{proof} 
  Straightforward substitution into equation \eqref{Generealized_eigenpair_system} shows that the quantities
  \begin{equation*}
    \nodeweight_u:=\frac{|\xi(u)|}{2\|f\|_{\infty}}
    \qquad
    \edgeweight:=\frac{|\Xi(u,v)|+|\Xi(v,u)|}{\|\grad f\|_{\infty}}    
  \end{equation*}
  are the desired admissible densities. The inverse implication follows by inverting the above definitions of $\nodeweight$ and $\edgeweight$.
\end{proof}
\begin{remark}
    
We remark the an analogue linearization of the graph $p$-Laplacian eigenvalue problem, when $p\in (2,\infty)$, has been used in \cite{deidda2024_spec_energy} to compute $p$-eigenpairs as limits of sequences of linear eigenpairs. Thus, in a similar way, the linearization of the $\infty$-eigenproblem in \cref{weighted_laplacian_equivalence} could be the starting point to investigate numerical methods aimed at computing generalized $\infty$-eigenpairs, we refer to \cite{DEIDDA_PHD} for a preliminary investigation. On the other hand, we refer to \cite{BungertInfNumerical} for a numerical investigation of the solutions of the $\infty$-limit eigenvalue equation.
\end{remark}
To continue towards the sought result we need some definitions. Let $(f,\Lambda)$ denote an $\infty$-eigenpair and assume $(\edgeweight, \nodeweight)$ to satisfy the condition of Proposition~\ref{weighted_laplacian_equivalence}.
We say that a node $u\in \nodeset$ is supported by $\edgeweight$, ($u\in\nodeset_{\edgeweight}$), if there exists an edge $(u,v)\in\edgeset$ such that $\edgeweight_{uv}>0$.
Observe that if $u\in\internalnodes$ but $u\not\in \nodeset_{\edgeweight}$, then necessarily $\nodeweight_u=0$ (recall that if $\nodeweight_u\neq0$, $|f(u)|=\|f\|_{\infty}$).
In particular, we write $u\in supp(\nodeset)$, if there exist at least one $(\edgeweight,\nodeweight)$ as in Proposition~\ref{weighted_laplacian_equivalence}, such that $u\in\nodeset_{\edgeweight}$.
Now we can prove that any generalized $\infty$-eigenpair can be regarded as a limit $\infty$-eigenpair up to considering the subraph of $\Gc$ induced by $\mbox{supp}(\nodeset)$.
\begin{theorem}\label{formulations_comparison}
  Assume $(f,\Lambda)$ to be an $\infty$-eigenpair as in Definition \ref{Generalized_eigenpair}. If $u\in\mbox{supp}(\nodeset) \cap \internalnodes$, then $f$ satisfies the limit eigenvalue equation \eqref{limitin_inf_eigenvalue_eq} in $u$.
\end{theorem}
\begin{proof}
  Let $u\in\mbox{supp}(\nodeset)\cap \internalnodes  $ and let $(\edgeweight,\nodeweight)$ be a pair of admissible densities, such that $u\in\nodeset_{\edgeweight}$. 
  Assuming $f(u)>0$, the weighted eigenvalue equation $-\divg(\edgeweight\grad f)(u)=\Lambda \nodeweight_u f(u)$ can be written as:
  \begin{equation}\label{eq1_comparison}
    -\|\grad f\|_{\infty}
    \sum_{v\sim u}\edgeweight_{uv}\,\edgelength_{uv}
    \mathrm{sign}\big(\grad f(u,v)\big)
    =\Lambda \nodeweight_u f(u)
    =\frac{\|\grad f\|_{\infty}}{\|f\|_{\infty}} \nodeweight_u f(u)\,.
  \end{equation}
  We first consider the case $f(u)<\|f\|_{\infty}$. Then $\nodeweight_u=0$, which yields $\|\Opc{\grad f}^-(u)\|_{\infty}=\|\Opc{\grad f}^+(u)\|_{\infty}$ i.e. $\inflap f(u)=0$, and:
  \begin{equation*}
    \|\Opc{\grad f}(u)\|_{\infty}-\Lambda f(u)=\|\grad f\|_{\infty}-\frac{\|\grad f\|_{\infty}}{\|f\|_{\infty}}f(u)>0\,.
  \end{equation*}
  If instead $f(u)=\|f\|_{\infty}$, for any $v\sim u$ we get $\grad f(u,v)\leq 0$ and thus $\nodeweight_u\neq0$ (otherwise, by hypotheses, \eqref{eq1_comparison}  is not satisfied). 
  Then, eq.~\eqref{eq1_comparison} reads:
  \begin{equation*}
    \|\grad f\|_{\infty}
    \Big(\sum_{v\sim u}\edgeweight_{uv}\edgelength_{uv}\Big)
    =\frac{\|\grad f\|_{\infty}}{\|f\|_{\infty}} \nodeweight_u \|f\|_{\infty}\,,
  \end{equation*}
  from which we obtain $\sum_{v\sim u}\edgeweight_{uv}\edgelength_{uv}=\nodeweight_u$. Replacing this last equality  back in $\eqref{eq1_comparison}$ we find:
  \begin{equation*}
    \begin{cases}
      \|\Opc{\grad f}(u)\|_{\infty}=\|\grad f\|_{\infty}=\Lambda f(u)\,,\\
      \inflap f(u)=\|\Opc{\grad f}^-(u)\|_{\infty}=\|\grad f\|_{\infty}>0\,.        
    \end{cases}
  \end{equation*}
   The proof is concluded by considering the cases $f(u)<0$ and $f(u)=0$, which can be proved analogously.
\end{proof}

We conclude this section by showing that there exist $\infty$-eigenpairs that are generalized critical points of $\rayl_{\infty}$ but that do not satisfy the $\infty$-eigenvalue equation \eqref{limitin_inf_eigenvalue_eq}.
By the same example, we also prove that there exist generalized $\infty$-eigenvalues concentrated between the first and second variational eigenvalues, which on the contrary never exist for the $p$-Laplacian when $p\in(1,\infty)$ \cite{DEIDDA2023_nod_dom}.

\begin{figure}[t]
  \begin{center}
    \begin{tikzpicture}[x={(\unitlength,0)},y={(0,\unitlength)}]

      \coordinate (Origin) at (0,0);
      \def\xscale{40}
      \def\yscale{40}
      \coordinate (Xone) at (\xscale,0);
      \coordinate (Yone) at (0,\yscale);

      \coordinate (N1) at ($(Origin)$);
      \coordinate (N2) at ($(Origin)+2.8*(Xone)$);
      \coordinate (N3) at ($(Origin)+4*(Xone)$);
      \coordinate (N4) at ($(Origin)+5.7*(Xone)$);
      
      \draw[black]
      (N1) node[shape=circle,draw,inner sep=0.5pt, fill=white] {\scriptsize{$B$}};
      \draw[black] 
      (N2) node[shape=circle,draw,inner sep=0.5pt, fill=white] {\scriptsize{$u_2$}} ;
      \draw[black]
      (N3) node[shape=circle,draw,inner sep=0.5pt, fill=white] {\scriptsize{$u_3$}};
      \draw[black]
      (N4) node[shape=circle,draw,inner sep=0.5pt, fill=white]{\scriptsize{$B$}};
      
      \begin{scope}[on background layer]
        \draw[black]
        (N1)--(N2) node[pos=0.5,sloped,above] {\small {$\edgelength_{B2}=1$}};;
        \draw[black]
        (N2)--(N3) node[pos=0.5,sloped,above] {\small {$\edgelength_{23}=3$}};
        \draw[black]
        (N3)--(N4)node[pos=0.5,sloped,above] {\small {$\edgelength_{3B}=2$}};;
      \end{scope}
    \end{tikzpicture}
  \end{center}
  \caption{A path graph admitting $\infty$-generalized eigenvalues that are not $\infty$-limit eigenvalues}
  \label{Fig-ex3}
\end{figure}

\begin{example}\label{Ex_Generalized}
  Consider the graph of figure \ref{Fig-ex3}:
  Note that the node farthest from the boundary is $u_2$ and $d(u_2,\boundary)=1/2+1/3=5/6$. 
  Then, the pair:
  \begin{equation*}
    f_1(u_2,u_3)=\left(\frac{5}{6},\,\frac{1}{2}\right)\;,\;\Lambda_1=\frac{6}{5}
  \end{equation*}
  is an $\infty$-eigenpair with 
  \begin{equation*}
    \left(\nodeweight_2,\, \nodeweight_3\right)
    =\left(\frac{6}{5},\,0\right)\;,
    \;\left(\edgeweight_{B2},
      \,\edgeweight_{23},
      \,\edgeweight_{3B}\right)
    =\left(0,\,\frac{2}{5},\,\frac{3}{5}\right)\,.
  \end{equation*}
  However, it is easy to verify that the following are also eigenpairs:
  \begin{equation*}
    \begin{aligned}
      &f_2(u_2,\,u_3)
        =\left(\frac{1}{6},\,-\frac{1}{6}\right),\;\Lambda_2
        =6\,,\;\left(\nodeweight_2,\, \nodeweight_3\right)
        =\left(3,\,3\right),
        \;\left(\edgeweight_{B2},
        \,\edgeweight_{23},
        \,\edgeweight_{3B}\right)
        =\left(0,\,1,\, 0\right)\,,\\
      &f(u_2,\,u_3)
        =\left(*,\frac{1}{2}\right),\,\Lambda
        =2,\; \left(\nodeweight_2,\nodeweight_3\right)
        =\left(0,2\right),
        \,\left(\edgeweight_{B2},
        \edgeweight_{23},
        \edgeweight_{3B}\right)
        =\left(0,\,0,\,1\right),
        *\in\left[\frac{1}{6},\frac{1}{2}\right]\!,
    \end{aligned}
  \end{equation*}
  where $*\in{\left[\frac{1}{6},\frac{1}{2}\right]}$ means that any value in $[1/6,1/2]$ produces an eigenfunction corresponding to the eigenvalue $\Lambda=2$.
  Note that $\Lambda_2=R_2=6$ is the second variational eigenvalue while $\Lambda_1<\Lambda=2<\Lambda_2$. 
  Moreover, it is worth noting that the pair $(f,\Lambda)$ does not satisfy~\eqref{limitin_inf_eigenvalue_eq}.
  Indeed, to get a solution of~\eqref{limitin_inf_eigenvalue_eq}
  $f(u_2)$ should be determined in such a way as to get:
  \begin{equation*}
    \inflap f(u_2)=0,\quad \|\Opc{\grad f}(u_2)\|_{\infty}-2f(u_2)\geq 0
  \end{equation*}
  or
  \begin{equation*}
    \inflap f(u_2)\geq 0,\quad \|\Opc{\grad f}(u_2)\|_{\infty}-2f(u_2)= 0\,.
  \end{equation*}
  In the first case, $\inflap f(u_2)=0$ implies:
  \begin{equation*}
    f_2(u_2)
    =\frac{3}{8}\;\Rightarrow\;\|\Opc{\grad f_2}(u_2)\|_{\infty}-2f(u_2)
    =\frac{3}{8}-\frac{6}{8}<0\,.
  \end{equation*}
  In the second case, the equality $\|\Opc{\grad f}(u_2)\|_{\infty}-2f(u_2)= 0$ implies $f_2(u_2)=\frac{3}{10}$ from which we get:
  \begin{equation*}
    \inflap f_2(u_2)=\frac{3}{10}-\frac{3}{5}<0\,.    
  \end{equation*}
\end{example}

\section{Conclusions and future directions}
\begin{NEW}
We studied the graph $\infty$-Laplacian eigenvalue problem on a graph, comparing different notions of $\infty$-eigenpairs. In particular, we have proved that any limit of $p$-Laplacian eigenpairs is a solution of the $\infty$-limit eigenvalue equation (\cref{limitin_infty_eigenvalue_theorem}), and that any solution of the $\infty$-limit eigenvalue equation is a generalized $\infty$-eigenpair (\cref{solutions_of_inf_eigenvalue_eq_are_generalized_inf_eigenpairs}). Additionally, by studying the topological properties of the $\infty$-eigenpairs, we have related the variational $\infty$-eigenvalues to the packing problem of the graph (\cref{Clarke_variational_eigen_characterization} and \cref{cor:lower-upper}).  

 We would like to note that the relations between the $\infty$-eigenpairs and the topology of the domain that we proved in the discrete setting hold also in the continuous setting \cite{Lind2, Lind3}. Moreover, the three different formulations of the $\infty$-eigenpairs are different in the continuous setting as well \cite{bungert2021eigenvalue, Lind2, Lind3, hynd2013nonuniqueness}. In particular, the inclusion 

    $$\left\{\begin{array}{cc}
         &\text{Limit of }\\
         &\plap\text{-eigenpairs}
    \end{array}\right\}
    \subsetneq
    \left\{\begin{array}{cc}
         & \text{Solutions of the} \\
         & \text{limit eigenvalue eq.}
    \end{array} \right\} 
    $$
    is known to hold also in the continuous setting \cite{Lind3, hynd2013nonuniqueness}. On the other hand, to the best of our knowledge, the second inclusion 
    $$
    \left\{\begin{array}{cc}
         & \text{Solutions of the} \\
         & \text{limit eigenvalue eq.}
    \end{array}\right\}
    \subsetneq
    \left\{\begin{array}{cc}
         & \text{Generalized } \\
         & \infty\text{-eigenpairs }
    \end{array}\right \}\,.$$
  is not known in the continuous setting.
  
    An open problem is also the study of the convergence of the discrete $\infty$-spectrum to the continuous $\infty$-spectrum on geometric random graphs. In this case the graph is given by a cloud of points obtained by sampling a ground truth measure and the study of convergence of the discrete spectrum to the continuous one would allow one to approximate the topological properties of the continuous domain in terms of information derived from the spectrum of the discrete $\infty$-Laplacian. Results in this direction have been presented in \cite{trillos2018variational, trillos2020error} for the linear Laplacian operator. In the same setting, different authors have also studied convergence of $\ell_p$ and $\ell_\infty$ based semi-supervised solutions \cite{calder2018game,calder2019consistency, slepcev2019analysis, roith2023continuum}. However, we are not aware of results about the convergence of the spectrum of nonlinear operators on geometric random graphs. 

     Finally, we note that the graph $p$-Laplacian operator can be given different formulations. Indeed, besides the variational graph $p$-Laplacian operator that we considered in \cref{discrete_plap-def}, it is also possible to consider the game theoretic $p$-Laplacian operator \cite{calder2018game, manfredi2015nonlinear, Elmoataz1} . This $p$-Laplacian operator is a linear combination of the 1 or 2 Laplacian and the $\infty$ one. Both the variational and game theoretic $p$-Laplacians can be interpreted as discretizations of the analogous continuous $p$-Laplacian operator. However, while the discrete variational $p$-Laplacian arises considering variational problems of the functional $f\to \|\grad f\|_p$, the game theoretic $p$-Laplacian does not have a variational formulation. For this reason, only the variational $p$-Laplacian is usually studied from a spectral point of view. Indeed critical points and values of functionals of the type “$\|\grad f\|_p/\|f\|$", where possibly different norms at the denominator can be considered, are of interest for different applications including spectral decompositions \cite{gilboa2018nonlinear, burger2016spectral} and clustering problems \cite{chang2016spectrum, Bhuler}. The investigation of the spectrum of the game theoretic $p$-Laplacian operator and its possible applications are currently an open problem.
 

\end{NEW}

\bibliographystyle{siamplain}
\bibliography{strings.bib, references.bib}
\end{document}